\documentclass[11pt]{amsart}
\usepackage{hyperref}
\usepackage{amsfonts}
\usepackage{amsmath}
\usepackage{amssymb}
\usepackage{amscd}
\usepackage{graphics}
\usepackage{latexsym}
\usepackage{color}
\usepackage{epsfig}
\usepackage{amsopn}
\usepackage{xypic}
\usepackage{mathrsfs}
\usepackage{array}

\theoremstyle{plain}
\newtheorem{lemma}{Lemma}[section]
\newtheorem*{theorem*}{Theorem}
\newtheorem*{lemma*}{Lemma}
\newtheorem*{proposition*}{Proposition}
\newtheorem*{conjecture*}{Conjecture}
\newtheorem*{corollary*}{Corollary}
\newtheorem*{problem*}{Problem}
\newtheorem{theorem}[lemma]{Theorem}

\newtheorem{corollary}[lemma]{Corollary}
\newtheorem{proposition}[lemma]{Proposition}
\newtheorem{problem}[lemma]{Problem}

\theoremstyle{definition}
\newtheorem{definition}[lemma]{Definition}
\newtheorem{example}[lemma]{Example}
\newtheorem{remark}[lemma]{Remark}

\newcommand{\fto}[1]{\stackrel{#1}{\to}}

\newcommand{\Z}{\mathbb{Z}}

\newcommand{\C}{\mathbb{C}}
\newcommand{\Q}{\mathbb{Q}}
\newcommand{\R}{\mathbb{R}}

\newcommand{\OO}{\mathcal{O}}
\newcommand{\te}{\otimes}

\newcommand{\opt}{\mathrm{opt}}
\newcommand{\cI}{\mathcal{I}}
\newcommand{\cF}{\mathcal F}
\newcommand{\cQ}{\mathcal Q}
\newcommand{\cA}{\mathcal A}
\newcommand{\cM}{\mathcal M}
\newcommand{\cP}{\mathcal P}

\renewcommand{\cD}{\mathcal{D}}
\newcommand{\rH}{\mbox{H}}

\newcommand{\cZ}{\mathcal{Z}}

\renewcommand{\P}{\mathbb{P}}
\newcommand{\PP}{\mathbb{P}}

\DeclareMathOperator{\ch}{ch}

\DeclareMathOperator{\Aut}{Aut}

\DeclareMathOperator{\Hom}{Hom}

\DeclareMathOperator{\Pic}{Pic}

\DeclareMathOperator{\rk}{rk}

\DeclareMathOperator{\Ext}{Ext}
\DeclareMathOperator{\Supp}{Supp}

\DeclareMathOperator{\coh}{coh}

\begin{document}

\date{\today}
\author[I. Coskun]{Izzet Coskun}
\author[J. Huizenga]{Jack Huizenga}
\address{Department of Mathematics, Statistics and CS \\University of Illinois at Chicago, Chicago, IL 60607}
\email{coskun@math.uic.edu}
\email{huizenga@math.uic.edu}
\thanks{During the preparation of this article the first author was partially supported by the NSF CAREER grant DMS-0950951535, and an Alfred P. Sloan Foundation Fellowship and the second author was partially supported by a National Science Foundation Mathematical Sciences Postdoctoral Research Fellowship}

\title[Interpolation for monomial schemes]{Interpolation, Bridgeland stability and monomial schemes in the plane}

\begin{abstract}
Given a zero-dimensional scheme $Z$, the higher-rank interpolation problem asks for the classification of slopes  $\mu$ such that there exists a vector bundle $E$ of slope $\mu$ satisfying $H^i(E \otimes \cI_Z)=0$ for all $i$. In this paper, we solve this problem for all zero-dimensional monomial schemes in $\PP^2$. As a corollary, we obtain detailed information on the stable base loci of $\Theta$-divisors on the Hilbert scheme of points on $\PP^2$.  We prove the correspondence between walls in the Bridgeland stability manifold and walls in the Mori chamber decomposition of the effective cone conjectured in \cite{ABCH} for monomial schemes.  We determine the Harder-Narasimhan filtration of ideal sheaves of monomial schemes for suitable Bridgeland stability conditions and, as a consequence, obtain a new resolution better suited for cohomology computations than other standard resolutions such as the minimal free resolution. 
\end{abstract}

\maketitle

\newcommand{\spacing}[1]{\renewcommand{\baselinestretch}{#1}\large\normalsize}
 \setcounter{tocdepth}{1}
\tableofcontents

\section{Introduction}
Interpolation problems form the technical core of many central questions in algebraic geometry, ranging from the Nagata-Harbourne-Hirschowitz conjecture to the construction of theta divisors on moduli spaces. In this paper, we solve the higher-rank interpolation problem for all zero-dimensional monomial schemes in $\PP^2$. Before we state our main theorem on interpolation, we introduce some terminology. 

\begin{definition}
A vector bundle $E$ on $\P^2$ \emph{satisfies interpolation} for a zero-dimensional scheme $Z\subset \P^2$ if $E\te \cI_Z$ is acyclic, so that $H^i(E\te \cI_Z)=0$ for all $i$.
\end{definition}

\begin{problem}[The higher-rank interpolation problem for zero-dimensional schemes]
Fix a zero-dimensional scheme $Z\subset \P^2$.  Determine the rational numbers $\mu$ for which there exists a vector bundle of slope $\mu$ that satisfies interpolation for $Z$.
\end{problem}

If $E$ satisfies interpolation for $Z$, then $\chi(E\te \cI_Z) = 0$.  By the Riemann-Roch formula, the discriminant $\Delta(E)$ of such a vector bundle is determined by its slope.  The interpolation problem is, therefore, equivalent to classifying all pairs of rational numbers $(\mu,\Delta)$ which are the invariants of a vector bundle with interpolation for $Z$.

Let $Z$ be a zero-dimensional monomial scheme. The block diagram $D$ associated to $Z$ records the monomials in $\C[x,y]/\cI_Z$ (see \S \ref{sec-prelim} for basic properties of monomial schemes). We index the rows from {\em bottom to top} and the columns from {\em left to right}. Let $h_j$ (respectively, $v_j$) denote the number of boxes in the $j$th row (respectively, column). Let $r(D)$ and $c(D)$ denote the number of rows and columns in $D$.  The {\em $k$th horizontal slope} $\mu_k$ and the {\em $i$th vertical slope} $\mu_i'$ are defined by 
$$\mu_k = \frac{1}{k} \sum_{j=1}^k (h_j + j-1) - 1, \ \  \ \ \mu_i'= \frac{1}{i} \sum_{j=1}^i (v_j + j-1) -1.$$ 
Let the {\em slope} $\mu(Z)$ of a zero-dimensional  monomial scheme with  block diagram $D$ be defined by 
$$\mu(Z) = \max_{{{1 \leq k \leq r(D)} \atop {1 \leq i \leq c(D)}}} \{\mu_k, \mu_i'\}.$$  

\begin{example}
The first diagram in Figure \ref{figure-slope} is the block diagram associated to the scheme defined by the monomials $x^7, x^6y, x^2y^3, xy^4, y^5$. The next two diagrams in Figure \ref{figure-slope} compute all the horizontal and vertical slopes.  For example, to compute horizontal slopes, place $h_j+j-1$ checkers in row $j$ for each $j$.  Then the $k$th horizontal slope $\mu_k$ is one less than the average number of checkers per row in the first $k$ rows.  In this example, the maximum slope $\mu(Z)$ is given by the horizontal slope $\mu_3= \frac{19}{3}$. 

\begin{figure}[htbp]
\begin{center}
\input{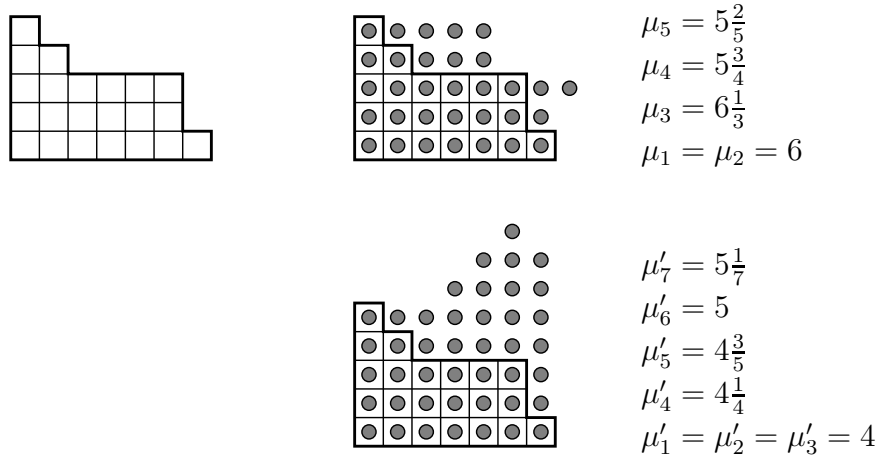}
\end{center}
\caption{Calculating the destabilizing object for a monomial scheme.}
\label{figure-slope}
\end{figure}
\end{example}

Our main theorem solves the higher-rank interpolation problem for monomial schemes.

\begin{theorem}\label{thm-intro-interpolate}
Let $Z$ be a zero-dimensional monomial scheme. There exists a vector bundle $E$ with slope $\mu \in \Q$ satisfying interpolation for $Z$ if and only if $\mu \geq \mu(Z)$.
\end{theorem}

More precisely, let $\zeta$ be the Chern character of a bundle $E$ satisfying interpolation for $Z$.  Then there are prioritary bundles with  Chern character $\zeta$,  and the general such bundle has interpolation for $Z$.  Furthermore, if there exist stable bundles with invariants $\zeta$, then the general such bundle also has interpolation for $Z$.

Theorem \ref{thm-intro-interpolate} has consequences for interpolation with respect to arbitrary zero-dimensional schemes.  By \cite[Theorem 15.17]{Eisenbud2}, there is a flat one-parameter torus action that specializes a homogeneous ideal of the polynomial ring to its generic initial ideal. Let $Y$ be an arbitrary zero-dimensional scheme and let $Z$ be the monomial scheme defined by the generic initial ideal of $Y$. Hence, we have a flat family $\lambda \cdot Y$ parameterized by $\C^*$ specializing to $Z$. Let $E$ be a vector bundle having interpolation for $Z$. Since being acyclic is an open condition in flat families, we conclude that $H^i(E \otimes \cI_{\lambda \cdot Y})=0$ for all $i$ and general $\lambda \in \C^*$. Consequently, $H^i( \lambda^* E \otimes \cI_Y) = 0 $ for all $i$, and we obtain the following corollary.

\begin{corollary}
Let $Y$ be a zero-dimensional subscheme of $\PP^2$, and let $Z$ be the scheme defined by the generic initial ideal of $Y$.  If $\mu\geq \mu(Z)$, then there exists a vector bundle $E$ with slope $\mu$ having interpolation for $Y$.

\end{corollary}

The main new ingredient that allows us to prove  Theorem \ref{thm-intro-interpolate} is Bridgeland stability. In general, computing the cohomology groups $H^i(E \otimes \cI_Z)$ is a difficult problem. The standard technique replaces $\cI_Z$ with a conveniently chosen quasi-isomorphic  complex of  coherent sheaves.  The main question is then to determine which quasi-isomorphic complex is the most convenient. Unfortunately, standard resolutions of $\cI_Z$ such as the minimal free resolution rarely help compute $H^i(E \otimes \cI_Z)$. For example, for a general collection of points, \cite{HuizengaPaper2} proved that the natural resolution of the ideal sheaf is in terms of exceptional bundles. The main philosophy of this paper is that Bridgeland stability determines the optimal complex for computing the cohomology groups $H^i(E \otimes \cI_Z)$.

More precisely, Bridgeland \cite{Bridgeland}, Arcara-Bertam \cite{ArcaraBertram}, and Bayer-Macr\`{i} \cite{BayerMacri} have constructed Bridgeland stability conditions for $\PP^2$ parameterized by an upper half-plane $(s,t)$, $t>0$.  The wall destabilizing $\cI_Z$ is a semi-circular wall with center along the negative real axis \cite{ABCH} (see \S \ref{sec-prelim} for details on Bridgeland stability). Let $A$ be the object destabilizing $\cI_Z$. Then we get an exact sequence $$0 \rightarrow A \rightarrow \cI_Z \rightarrow B \rightarrow 0,$$ where $A$ and $B$ are semi-stable objects with respect to the stability conditions  along the wall. 

 We conjecture that if $E$ is a minimal slope bundle satisfying interpolation for $Z$, then $E \otimes \cI_Z$ is acyclic because $E \otimes A$ and $E \otimes B$ are acyclic. Thus, the Bridgeland decomposition allows one to solve the interpolation problem for $Z$ assuming the solutions of the interpolation problems for $A$ and $B$ are known.   In this paper, we will show that for complete intersection schemes (see \S \ref{sec-ci}) and monomial schemes (see \S \ref{sec-interpolate}) this philosophy works perfectly and solves the interpolation problem.

We now describe the destabilizing object for a monomial scheme in greater detail. For concreteness, assume that the maximum defining the slope $\mu(Z)$ is achieved by $\mu_k$.  We will prove that the subobject destabilizing $\cI_Z$ is  $\cI_W (-k),$ where $W$ is the monomial scheme corresponding to the subdiagram lying above the $k$th row of the block diagram $D$. By \cite{ABCH}, the center of the Bridgeland wall corresponding to $\cI_W (-k)$ is the point $(s,0)$ with 
$$s= - \frac{\deg Z - \deg W}{k} - \frac{k}{2} = - \frac{1}{k} \sum_{j=1}^k h_j - \frac{k}{2} = - \mu(Z) - \frac{3}{2}.$$ 
We summarize this discussion in the following theorem.

\begin{theorem}\label{thm-intro-Bridgeland}
Let $Z$ be a zero-dimensional monomial scheme in $\PP^2$ of length $n$ with block diagram $D$. Swapping $x$ and $y$ if necessary, assume that $\mu_k$ is the slope giving the maximum in the definition of $\mu(Z)$. Let $W$ be the monomial scheme associated to the subdiagram of $D$ lying above the $k$th row. Then the ideal sheaf  $\cI_Z$ is destabilized by $\cI_W(-k)$ along the semi-circular wall with center $s= -\mu(Z) - \frac{3}{2}$ and radius $\sqrt{s^2-2n}$.
\end{theorem}

Let $\PP^{2[n]}$ denote the Hilbert scheme parameterizing zero-dimensional schemes of length $n$ in the projective plane $\PP^2$. The main application of our theorems is to the study of stable base loci of $\Theta$-divisors on $\PP^{2[n]}$.   Recall that $\PP^{2[n]}$ is a smooth, irreducible, projective variety of dimension $2n$ containing the locus of  distinct collections of $n$ points as a dense open subset \cite{Fogarty1}. The Hilbert scheme admits a natural morphism called the {\em Hilbert-Chow morphism} $h: \PP^{2[n]} \rightarrow \PP^{2(n)}$ to the symmetric product $\PP^{2(n)} = (\PP^2)^n / \mathfrak{S}_n$, sending a  zero-dimensional scheme of length $n$ to its support weighted with multiplicity. The map $h$ is a crepant resolution of singularities \cite{Fogarty1}. Furthermore, $\PP^{2[n]}$ is a log Fano variety, hence a Mori dream space \cite[Theorem 2.5]{ABCH}. Let $B$ be the exceptional divisor of $h$ parameterizing non-reduced schemes. Let $H = h^* \OO_{\PP^{2(n)}}(1)$ be the pullback of the ample generator from the symmetric product. Geometrically, $H$ is the class of the divisor of schemes whose support intersects a fixed line.  Fogarty \cite{Fogarty2} proves that $$\Pic (\PP^{2[n]}) \cong \Z H \oplus \Z \frac{B}{2}.$$ Hence, we can express the class of every $\Q$-divisor on $\PP^{2[n]}$ as a linear combination of $H$ and $B$. 

To understand the stable base locus decomposition of the cone of effective divisors, one naturally considers $\Theta$-divisors.  Let $E$ be a vector bundle which satisfies interpolation for some scheme $Z'$ of length $n$.  Then there is an effective divisor $$D_E := \{Z\in \P^{2[n]}:H^0(E\te \cI_Z)\neq 0\},$$ and $Z'$ is not in the base locus of $D_E$.  By a simple Grothendieck-Riemann-Roch calculation (see \cite{ABCH}) $$[D_{E}] = c_1(E) H - \frac{r(E)}{2} B,$$ which is a multiple of the class $$\mu(E)H - \frac{1}{2}B.$$ As a corollary of Theorem \ref{thm-intro-interpolate}, we determine when a monomial scheme is in the stable base locus of a linear system on $\PP^{2[n]}$.

\begin{corollary}\label{intro-corollary-sbl}
Let $Z$ be a zero-dimensional monomial scheme. Then $Z$ is in the stable base locus of the linear system $|\mu H - \frac{B}{2}|$ if and only if $\mu < \mu(Z)$. 
\end{corollary}

We can generalize Corollary \ref{intro-corollary-sbl} to arbitrary zero-dimensional schemes by passing to the generic initial ideal.

\begin{corollary}\label{intro-cor-gen}
Let $Y$ be a zero-dimensional scheme and let $Z$ be the scheme defined by the generic initial ideal of $Y$. Then $Y$ is not in the stable base locus of the linear system $|aH - \frac{B}{2}|$ if $a \geq \mu(Z)$. 
\end{corollary}

The paper \cite{ABCH} conjectures that a scheme $Z$ is in the stable base locus of the divisors $a H - \frac{B}{2}$  for $a < \mu$  if and only if the ideal sheaf $\cI_Z$ is destabilized at the Bridgeland wall with center $(s,0)$, where $s = - \mu - \frac{3}{2}$. This is checked in \cite{ABCH} for $n \leq 9$ and for all $n$ provided that $\mu \geq  \frac{n-1}{2}$. As an immediate corollary of Theorems \ref{thm-intro-interpolate} and \ref{thm-intro-Bridgeland}, we conclude that the conjecture holds for monomial schemes.

\begin{corollary}
Let $Z$ be a monomial scheme in $\PP^2$. The conjectural correspondence between Bridgeland walls and Mori walls holds for $Z$.
\end{corollary}

The organization of this paper is as follows. In \S \ref{sec-prelim}, we will review the basics of  monomial schemes, vector bundles on $\PP^2$ and Bridgeland stability. In \S \ref{sec-interpolation}, we will prove general theorems concerning interpolation. Section \ref{sec-correspondence} gives a theoretical explanation for the correspondence between Mori walls and Bridgeland walls conjectured in \cite{ABCH}. In \S \ref{sec-ci}, we will solve the interpolation problem for all complete-intersection zero-dimensional schemes. This will demonstrate the general philosophy  in a simple example. Section \ref{sec-monomialObjects} serves as the roadmap for the proof of Theorem \ref{thm-intro-interpolate}. We will describe the inductive structure and introduce the  monomial objects that occur as subobjects and quotient objects in the Bridgeland destabilizing sequences. The next several sections will substantiate the claims in \S \ref{sec-monomialObjects}. In \S \ref{sec-invariants}, we collect numerical invariants associated to monomial objects. In \S \ref{sec-Gieseker}, we check the Gieseker stability of monomial objects. In \S \ref{sec-nesting}, we determine the Bridgeland walls where the monomial objects are destabilized. Finally, in \S \ref{sec-interpolate}, we solve the interpolation problem for monomial schemes.

\subsection{Acknowledgements} We would like to thank Daniele Arcara, Arend Bayer, Aaron Bertram, Joe Harris, and Emanuele Macr\`{i} for many useful discussions.

\section{Preliminaries}\label{sec-prelim}
In this section, we collect basic facts concerning monomial schemes, stable vector bundles on $\PP^2$, and Bridgeland stability.

\subsection{Monomial schemes} We refer the reader to \cite{Eisenbud} for details on monomial schemes.

A {\em monomial scheme} is a scheme whose homogeneous ideal is generated by monomials for some choice of coordinates. In this paper, we will be interested in zero-dimensional monomial schemes in $\PP^2$.  Such a scheme $Z$ is generated by a set of monomials $$x^{a_1}, x^{a_2}y^{b_2}, \dots, y^{b_r},$$ where $a_1> \cdots> a_{r-1}$ and $b_2< \cdots <b_r$.

A {\em block diagram} is a left-justified diagram consisting of finitely many rows of finitely many consecutive boxes such that the number of boxes in each row is non-increasing as we proceed from bottom to top. The box at the lower left corresponds to the monomial $1$. If a box represents the monomial $x^a y^b$, then the box to its immediate right represents the monomial $x^{a+1}y^b$ and the box immediately above represents the monomial $x^a y^{b+1}$.  We can represent $Z$ by a block diagram $D_Z$ which records the monomials in the quotient $$\frac{\C[x,y]}{(x^{a_1}, x^{a_2}y^{b_2}, \dots, y^{b_r})}.$$ The corresponding block diagram has $b_{i+1} - b_{i}$ rows of  $a_i$ boxes for $1 \leq i \leq r-1$. The length of the scheme $Z$ is the total number of boxes in the diagram. Figure \ref{figure-box} depicts the block diagram associated to the scheme generated by the monomials $x^5, x^4 y^2, x^3y^3$ and $y^5$.

\begin{figure}[htbp]
\begin{center}
\input{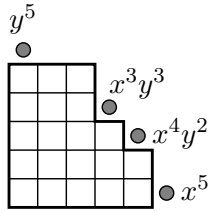}
\end{center}
\caption{Representing a monomial scheme by a block diagram.}
\label{figure-box}
\end{figure}

Conversely, given a block diagram $D$, we can associate a monomial scheme $Z(D)$ to it such that $D = D_{Z(D)}$. Suppose that $D$ has $n_j$ rows of length $a_j$ for $1 \leq j \leq r-1$. Set $a_r=0$ and set $b_i = \sum_{j=1}^{i-1} n_j$ for $1 \leq i \leq r$. The monomial scheme $Z(D)$ is the scheme generated by the monomials $x^{a_i} y^{b_i}$ for $1 \leq i \leq r$. 

The minimal free resolution of a monomial ideal is easy to determine from the set of generators. If $Z$ is generated by the monomials $x^{a_1}, x^{a_2}y^{b_2}, \dots, y^{b_r},$ then the minimal free resolution has the form $$0 \rightarrow \bigoplus_{i=1}^{r-1} \OO(-a_i-b_{i+1}) \xrightarrow{M} \bigoplus_{i=1}^r \OO (-a_i -b_i) \rightarrow \cI_Z \rightarrow 0,$$ where $M$ is the $r \times (r-1)$ matrix  with entries $m_{i,i} = y^{b_{i+1} - b_i}$, $m_{i+1, i} = - x^{a_i - a_{i+1}}$ and $m_{i,j} =0$ otherwise. 
For example, the minimal free resolution of the monomial scheme in Figure \ref{figure-box} is given by 
$$ 0\to \OO(-7)^2 \oplus \OO(-8) \xrightarrow{\left(\begin{array}{ccc} y^2 & 0 & 0 \\ -x & y & 0 \\ 0 & -x & y^2 \\ 0 & 0 & -x^3 \end{array} \right)} \OO(-5) \oplus \OO(-6)^2 \oplus \OO(-5) \xrightarrow{\displaystyle(x^5, x^4y^2, x^3y^3, y^5)}\cI_Z \to 0.$$ 
Observe that every Betti diagram of a zero-dimensional scheme occurs as the Betti diagram of a zero-dimensional monomial scheme \cite{Eisenbud}.

\subsection{Coherent sheaves on $\PP^2$ and stability}\label{sec-prelim-stable}

We refer the reader to \cite{LePotierLectures} and \cite{HuybrechtsLehn} for general background material.  All sheaves in this paper will be coherent.

Let $E$ be a coherent sheaf on $\P^2$.  The dimension $\dim E$ of $E$ is the dimension of the support $\Supp(E)$.  We say $E$ is \emph{pure of dimension $d$} if $\dim E = d$ and every non-trivial coherent subsheaf $F\subset E$ has $\dim F = d$.  

If $E$ has dimension $d$, the Hilbert polynomial $P_E(m) = \chi(E(m))$ of $E$ is of the form $$P_E(m)=\alpha_d \frac{m^d}{d!} + O(m^{d-1}).$$ The \emph{reduced Hilbert polynomial} $p_E$ is defined by $$p_E = \frac{P_E}{\alpha_d}.$$ Then $E$ is \emph{(Gieseker) semistable} (resp. \emph{stable}) if $E$ is pure and for every nontrivial $F\subset E$ we have $p_F\leq p_E$ (resp. $<$), where polynomials are compared for all sufficiently large $m$.  

In the main case of interest, $E$ is pure of dimension $2$ with rank $r>0$.  While the numerical invariants $(r = \ch_0,c_1 = \ch_1,\ch_2)$ can always be used for any coherent sheaf, it is useful to instead use the invariants $(r,\mu,\Delta)$ whenever $r \neq 0$, where the (Mumford) slope $\mu$ and discriminant $\Delta$ are defined by $$\mu = \frac{c_1}{r} \qquad \Delta = \frac{1}{2} \mu^2 - \frac{\ch_2}{r}.$$ Note that if a Chern character $\xi$ is scaled by a number $\lambda$, then $\lambda \xi$ has the same slope and discriminant as $\xi$.  The discriminant also satisfies the identity $\Delta(E\te F) = \Delta(E) + \Delta(F)$.  The sheaf $E$ is stable if and only if every nontrivial subsheaf $F\subset E$ has $\mu(F)\leq \mu(E)$, with $\Delta(F) >\Delta(E)$ in case of equality.

In terms of the slope and discriminant, the Riemann-Roch formula becomes particularly simple.  For a sheaf $E$ of nonzero rank, we have $$\chi(E) = r(P(\mu)-\Delta),$$ where $$P(m) = P_{\OO_{\P^2}}(m) = \frac{1}{2} (m^2+3m+2)$$ is the Hilbert polynomial of the trivial sheaf.  We also define the Euler characteristic $\chi(E,F)$ of a pair of coherent sheaves by $$\chi(E,F) = \sum_{i=0}^2 (-1)^i \dim \Ext^i(E,F).$$ When both sheaves have nonzero rank, this invariant is computed by a Riemann-Roch formula $$\chi(E,F) = r(E)r(F) (P(\mu(F)-\mu(E))-\Delta(E)-\Delta(F)).$$ Since Euler characteristics depend only on Chern characters,  the Euler characteristic induces a pairing $K(\P^2) \times K(\P^2)\to \Z$.  By abuse of notation, we will write expressions such as $\chi(\zeta,\xi)$ or $\chi(\zeta,E)$ to denote this pairing.  The derived dual induces a homomorphism $K(\P^2)\to K(\P^2)$, so we will similarly write $\zeta^*$ for the dual Chern character.

Every coherent sheaf of rank $r\geq 1$ has a Harder-Narasiman filtration with respect to the Mumford slope. We will denote the maximal and minimal slopes in the filtration by $\mu_{\max}$ and $\mu_{\min}$, respectively.
 
 The Bogomolov inequality states that a stable sheaf of nonzero rank satisfies $\Delta \geq 0$.  A result of Drezet and Le Potier refines the Bogomolov inequality to classify the admissible numerical invariants of stable sheaves.  In other words, they classify the numerical invariants $\xi$ such that the moduli space $M(\xi)$ of stable coherent sheaves is nonempty.

\begin{theorem}\cite{Drezet,DLP,LePotierLectures}\label{stableClassification}
There is an explicit function $\delta : \Q\to \Q$ such that the moduli space $M(\xi)$ is positive-dimensional if and only if $\Delta(\xi) \geq \delta(\mu(\xi))$.
\end{theorem}

The cases where the space $M(\xi)$ is zero-dimensional (and precisely $1$ point, corresponding to an \emph{exceptional bundle}) can also be explicitly determined.

\subsection{Bridgeland stability}\label{sec-prelim-Bridge}
In this subsection,  we recall the definition of Bridgeland stability conditions and describe the chamber decomposition of the stability manifold of $\PP^2$. We refer the reader to \cite{ArcaraBertram}, \cite{ABCH}, \cite{BertramCoskun}, \cite{bridgeland:stable}, and \cite{Bridgeland} for more detailed information. 

Let $D^b(\PP^2)$ denote  the bounded derived category of coherent sheaves on $\PP^2$. Let $L$ denote the class of a line. A {\em Bridgeland stability condition} $\sigma$ on $\PP^2$ consists of a pair $\sigma= (\cA, \cZ)$ such that $\cA$ is the heart of a bounded $t$-structure on $D^b (\PP^2)$ and $\cZ : K (D^b (\PP^2)) \rightarrow \C$ is a homomorphism satisfying the following properties:
\smallskip

\noindent (1) (Positivity) For every non-zero object $E$ of $\cA$, $\cZ(E)$ lies in the semi-closed upper half-plane $\{ r e^{i \pi \theta} \ | \ r>0, 0 < \theta \leq 1\}.$ Writing $\cZ= -d(E) + i r(E)$, one may view this condition as two separate positivity conditions requiring $r(E) \geq 0$ and if $r(E) =0$,  then $d(E) > 0$.
\smallskip

\noindent (2) (Harder-Narasimhan property) For an object $E$ of $\mathcal{A}$, let the $\cZ$-slope of $E$ be defined by setting  $\mu(E) = d(E)/ r(E)$ with the understanding that $\mu(E) = \infty$ if $r(E)=0$. An object $E$ is called $\cZ$-stable (resp. semistable) if for every proper subobject $F$, $\mu(F) < \mu(E)$ (resp. $\leq$). The pair $(\cA, \cZ)$ is required to satisfy the Harder-Narasimhan property. Namely, every object of $\mathcal{A}$ has a finite filtration 
$$0 = E_0 \hookrightarrow E_{1} \hookrightarrow \cdots \hookrightarrow E_n = E$$ such that $F_i = E_i / E_{i-1}$ is $\cZ$-semi-stable and $\mu(F_i) > \mu(F_{i+1})$ for all $i$.
\smallskip

The category of coherent sheaves with the stability function $\cZ(E) = - \deg(E) + i  \rk(E)$ is not a Bridgeland stability condition on $\PP^2$ because $\cZ$ is zero on sheaves supported on points. The idea of Bridgeland, Arcara and Bertram is to fix this problem by tilting the category.  Given $s\in \R$, let $\cQ_s$ be the full subcategory of $\coh(\PP^2)$ consisting of torsion sheaves or sheaves $Q$ where $\mu_{\min}(Q)> s$ (where, as in \S \ref{sec-prelim-stable}, $\mu_{\min}$ denotes the minimum slope of a Harder-Narasimhan factor with respect to the Mumford slope).  Similarly, let $\cF_s$  be the full subcategory of $\coh(\PP^2)$ consisting of torsion free sheaves $F$ with $\mu_{\max}(F) \leq s$.
 
By \cite[Lemma 6.1]{Bridgeland}, each pair $(\cF_s,\cQ_s)$ of full subcategories satisfies the two properties
\begin{itemize}
\item[(a)] For all $F\in \cF_s$ and $Q \in \cQ_s$, $\Hom(Q,F) = 0$.
\item[(b)] Every coherent sheaf $E$ fits in a short exact sequence
$0 \rightarrow Q \rightarrow E \rightarrow F \rightarrow 0,$
where $Q \in \cQ_s$, $F\in \cF_s$ and the extension class are uniquely determined
up to isomorphism.
\end{itemize}
\noindent A pair of full subcategories $(\cF,\cQ)$ of an abelian category $\cA$ satisfying conditions (a) and (b) is called a {\em torsion pair}.  A torsion pair $(\cF, \cQ)$  
defines the heart of a new $t$-structure on $\cD^b(\cA)$  by setting \cite{tilting}
$$\cA_{(\cF, \cQ)}:= \{E \in \cD^b(\P^2) \ | \ \rH^{-1}(E) \in \cF, \rH^0(E) \in \cQ, \ \mbox{and}\ \rH^i(E) = 0 \ \mbox{otherwise}\}.$$
The natural exact sequence
$$0 \rightarrow \rH^{-1}(E)[1] \rightarrow E \rightarrow \rH^0(E) \rightarrow 0$$
for such an object of $\cA$ implies that the objects of the heart are all given by pairs of objects 
$F \in \cF$ and $Q \in \cQ$ together with an extension class in $\Ext_{\cA}^2(Q,F)$ \cite{tilting}.

\begin{definition}
Let $\cA_s$ be the heart of the $t$-structure on $\cD^b(\coh(\PP^2))$ obtained from 
the torsion-pair $(\cF_s, \cQ_s)$. 
Define a central charge by setting $$\cZ_{s,t}(E) = - \int_{\P^2} e^{-(s+it)L} \ch(E).$$
\end{definition}

The above formula for $\cZ_{s,t}(E)$ can be expanded out as a formula in terms of the Chern character of $E$.  If $E$ has nonzero rank, it is convenient to express the central charge in terms of the slope and discriminant of $E$.  Specifically, if $\xi = (r,\mu,\Delta)$ is a Chern character with $r\neq 0$, we have 
$$
\cZ_{s,t}(\xi) =r(\Delta-\frac{1}{2}(s+it-\mu)^2) = -\frac{1}{2}r((\mu-s)^2-t^2-2\Delta)+irt(\mu-s).
$$
The slope function $\mu_{s,t}$ is then given by $$\mu_{s,t}(\xi) = \frac{(\mu-s)^2-t^2-2\Delta}{2t(\mu-s)}.$$

\begin{theorem}[Bridgeland \cite{Bridgeland}, Arcara-Bertram \cite{ArcaraBertram}, Bayer-Macr\`i \cite{BayerMacri}]
 For each $s\in \R$ and $t > 0$,  the pair $(\cA_s, \cZ_{s,t})$ defines a Bridgeland stability condition on $\cD^b(\operatorname{coh}(\PP^2))$.
\end{theorem}

Fix a class $\xi$ in the numerical Grothendieck group. Then there exists a locally finite set of walls in the $(s,t)$-half plane, depending only on $\xi$, such that as the stability condition $\sigma = \sigma_{s,t}$ varies in a chamber, the set of $\sigma$-(semi)-stable objects of class $\xi$ does not change (\cite{Bridgeland}, \cite{BayerMacri}, \cite{BayerMacri2}). We call these walls {\em Bridgeland walls}.

Suppose $\xi,\zeta\in K(\P^2)\te \R$ are two linearly independent real Chern characters.  By contrast with Bridgeland walls, a \emph{potential Bridgeland wall} is a set in the $(s,t)$-half-plane of the form $$W(\xi,\zeta) = \{(s,t):\mu_{s,t}(\xi) = \mu_{s,t}(\zeta)\},$$ where $\mu_{s,t}$ is the slope associated to $\cZ_{s,t}$.  Bridgeland walls are always potential Bridgeland walls.  The \emph{potential Bridgeland walls for $\xi$} are all the potential walls $W(\xi,\zeta)$ as $\zeta$ varies in $K(\P^2)\te \R$.  If $E,F\in D^b(\P^2)$, we also write $W(E,F)$ as a shorthand for $W(\ch(E),\ch(F))$.

The potential walls $W(\xi,\zeta)$ can be easily computed in terms of the Chern characters $\xi = (r,c,d)$ and $\zeta = (r',c',d')$.  
\begin{enumerate} \item If $\mu(\xi) = \mu(\zeta)$ (where the Mumford slope is interpreted as $\infty$ if the rank is $0$) then the wall $W(\xi,\zeta)$ is the vertical line $s= \mu(\xi)$ (interpreted as the empty set when the slope is infinite).  
\item Otherwise, without loss of generality assume $\mu(\xi)$ is finite, so that $r\neq 0$.  The walls $W(\xi,\zeta)$ and $W(\xi,\xi+\zeta)$ are equal, so we may further reduce to the case where both $\xi$ and $\zeta$ have nonzero rank. Then we may encode $\xi = (r_1,\mu_1,\Delta_1)$ and $\zeta = (r_2,\mu_2,\Delta_2)$ in terms of slope and discriminant instead of $\ch_1$ and $\ch_2$.  The wall $W(\xi,\zeta)$ is the semicircle centered at the point $(s,0)$ with $$s = \frac{1}{2}(\mu_1+\mu_2)-\frac{\Delta_1-\Delta_2}{\mu_1-\mu_2}$$ and having radius $\rho$ given by $$\rho^2 = (s-\mu_1)^2-2\Delta_1.$$ If this expression is negative, we define the \emph{virtual radius} of $W(\xi,\zeta)$ by this formula, and the wall is empty.
\end{enumerate}

In the principal case of interest, the Chern character $\xi = (r,\mu,\Delta)$ has nonzero rank $r\neq 0$ and nonnegative discriminant $\Delta\geq 0$.  In this case, the potential walls for $\xi$ consist of a vertical wall $s=\mu$ together with two disjoint nested families of semicircles on either side of this line \cite{ABCH}.  Specifically, for any $s$ with $|s-\mu| > \sqrt{2\Delta}$, there is a unique semicircular potential wall with center $(s,0)$ and radius $\rho$ satisfying $$\rho^2 = (s-\mu)^2 - 2\Delta.$$ The semicircles are centered along the $s$-axis, with smaller semicircles having centers closer to the vertical wall.  Every point in the $(s,t)$-half-plane lies on a unique potential wall for $\xi$.  When $r>0$, only the family of semicircles left of the vertical wall is interesting, since an object $E$ with Chern character $\xi$ can only be in categories $\cA_s$ with $s<\mu$.  Similarly, for objects with $r<0$ only the family of semicircles to the right of the vertical wall is relevant.

Occasionally it will be important to consider potential walls for $1$-dimensional sheaves as well.  For a Chern character $\xi = cL+dL^2$ with $c\neq 0$, the potential Bridgeland walls are concentric semicircles centered at the point $(\frac{d}{c},0)$.  These semicircles foliate the upper half-plane.

\begin{figure}[htbp]
\begin{center}
\input{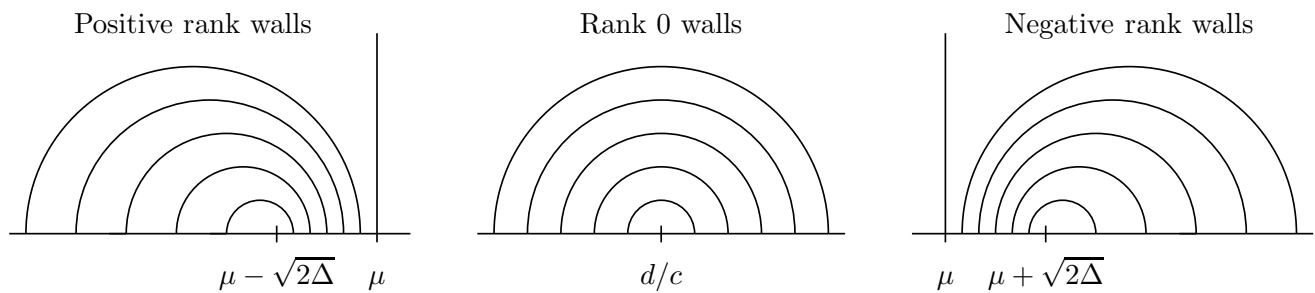}
\end{center}
\caption{Relevant potential Bridgeland walls for (i) a positive rank object with $\Delta\geq 0$; (ii) a rank $0$ object with Chern character $cL+dL^2$ and $c\neq 0$; (iii) a negative rank object with $\Delta\geq 0$. For positive (resp. negative) rank walls, the centers of the walls converge to $\mu - \sqrt{2\Delta}$ (resp. $\mu + \sqrt{2\Delta}$).}
\label{figure-walls}
\end{figure}

\section{Interpolation, stability, and prioritary bundles}\label{sec-interpolation}

In this section, we study the relation between stability of vector bundles on $\P^2$ and interpolation problems.  Our first main result shows that semistability is a natural condition to consider when studying interpolation problems.  We first need a simple lemma.

\begin{lemma}\label{interpOnePointLem}
Let $E$ be a torsion-free sheaf such that $H^1(E\te \cI_p)=0$ for a general point $p\in \P^2$.  Then 
\begin{enumerate}
\item $\mu(E)\geq 0$,
\item $H^1(E) = 0$, and
\item $H^2(E) = H^2(E\te \cI_p) = 0$.
\end{enumerate}
\end{lemma}
\begin{proof}
Since $p$ is a general point, there is an exact sequence of sheaves $$0\to E \te \cI_p \to E \to E\te \OO_p\to 0.$$ The long exact sequence for cohomology implies part (2) and the fact that $H^2(E) = H^2(E\te \cI_p)$ for a general $p$.  

As $p$ is general, $E$ is locally free of rank $r$ near $p$.  We find from $H^1(E\te \cI_p)=0$ that the restriction map $H^0(E)\to E_p$ is surjective, where $E_p$ is the fiber of $E$ at $p$.  Then the values of $r$ general sections $s_1,\ldots,s_r\in H^0(E)$ span the fiber $E_p$, so $\wedge_i s_i$ is a nonzero section of $\wedge^r E$.  The inclusion $$0\to \wedge^r E \to (\wedge^r E)^{**} = \det E$$ gives a section of $\det E$.  We conclude $c_1(\det E) = c_1(E)$ is effective, so $\mu(E)\geq 0$.

By  induction on the length $\ell$ of the Harder-Narasimhan filtration of $E$, we show that if $H^1(E\te \cI_p)=0$, then $H^2(E)=0$.  If $E$ is semistable, then $\mu(E) \geq 0$ and Serre duality implies that $H^2(E)=0$.

In the general case, let $\ell\geq 2$ and suppose the result is known for torsion-free sheaves with Harder-Narasimhan filtration of length at most $\ell-1$.  Let $0\subset E_1\subset \cdots \subset E_{\ell} = E$ be the Harder-Narasimhan filtration, with semistable quotients $E_i/E_{i-1}$ of decreasing slope.  Consider the exact sequence $$0\to E_1 \to E \to E/E_1 \to 0,$$ with $E_1$ semistable and $E/E_1$ torsion-free. Then $\mu(E_1) > \mu(E) \geq 0$, so $H^2(E_1) = 0$ and $H^2(E_1\te \cI_p) = 0$ for general $p$.  We find $H^1((E/E_1)\te \cI_p)=0,$ so by induction $H^2(E/E_1)=0$ and $H^2(E)=0$.
\end{proof}

\begin{remark}
In particular, the assumptions of Lemma \ref{interpOnePointLem} are satisfied if $E$ is a vector bundle and $H^1(E\te \cI_Z)=0$ for some fixed nonempty zero-dimensional scheme $Z$.
\end{remark}

\begin{theorem}\label{semistableThm}
Let $Z\subset \P^{2}$ be a nonempty zero-dimensional scheme, and suppose $E$ is a vector bundle such that $H^1(E\te \cI_Z)=0$.  If the slope of $E$ is minimal among all bundles with this property, then $E$ is semistable.
\end{theorem}
\begin{proof}
Suppose $E$ is not semistable.  Let $F\subset E$ be a maximal destabilizing subsheaf.  In the exact sequence $$0 \to F \to E \to Q\to 0,$$ the sheaf $F$ is semistable of slope $\mu(F)>\mu(E)$ and the sheaf $Q$ is torsion-free of slope $\mu(Q)<\mu(E)$.  A priori some of the singularities of $F$ or $Q$ could be supported at points in the support of $Z$.  If $g\in \Aut \P^2$ is an automorphism of $\P^2$, the sheaves $g^\ast E \te \cI_Z$ form a flat family over $\Aut \P^2$ since $E$ is locally free.  Then if $g\in \mathrm{Aut}\,\P^2$ is general, we have $H^1(g^\ast E \te \cI_Z)=0$ and the sheaves $g^\ast F,g^\ast Q$ will have singularities disjoint from $Z$.  Semistability of $E$ is equivalent to semistability of $g^\ast E$, so without loss of generality we assume the singularities of $F$, $Q$ do not meet $Z$.  This implies the sequence $$0\to F \te \cI_Z\to E\te \cI_Z \to Q\te \cI_Z \to 0$$ is exact.  We wish to show $H^1(Q\te \cI_Z)=0$, which will follow from our assumption on $E$ if we can show $H^2(F\te \cI_Z)=0$.

By Lemma \ref{interpOnePointLem}, we have $\mu(E) \geq 0$, so $\mu(F)>0$ and $H^2(F) = 0$ since $F$ is semistable.  We easily conclude $H^2(F\te \cI_Z)=0$ as in the proof of the lemma.

Finally, to obtain a contradiction, consider the (locally free) double dual $Q^{\ast\ast}$, with $\mu(Q^{\ast\ast})= \mu(Q)$.  Since $Q$ is torsion-free, there is an exact sequence $$0\to Q\to Q^{\ast\ast}\to T\to 0$$ where $T$ is $0$-dimensional with support disjoint from $Z$.  We conclude $H^1(Q^{\ast\ast}\te \cI_Z) = 0$, which violates the minimality assumption on $E$.
\end{proof}

While it is not at all obvious that the set of slopes of vector bundles with $H^1(E \te \cI_Z)=0$ actually has a minimum, analysis of the proof nevertheless yields the following result.

\begin{corollary}\label{assumeStableCor}
Let $Z\subset \P^2$ be a nonempty zero-dimensional scheme, and suppose $E$ is a vector bundle such that $H^1(E\te \cI_Z)=0$.  Then there is a stable bundle $E'$ of slope $\mu(E')\leq \mu(E)$ such that $H^1(E'\te \cI_Z)=0$.
\end{corollary}
\begin{proof}
If $E$ is not already stable, pass from $E$ to $Q^{**}$ as in the proof of the theorem.  Then apply the argument again to $Q^{**}$, and continue in this fashion.  The ranks of the bundles we obtain are decreasing, so this process must eventually stop, in which case the bundle we obtain is semistable.  By a Jordan-H\"older argument, we can further make the bundle stable.
\end{proof}

The previous results in this section merely concerned $H^1$-vanishing of the sheaf $E\te \cI_Z$.  For the remainder of the section we will focus on the stronger notion of acyclicity of $E\te \cI_Z$ instead, that is, on the interpolation problem for $Z$.  

In what follows, it will be useful to introduce an additional space of coherent sheaves.  A torsion-free coherent sheaf $E$ on $\P^2$ is \emph{prioritary} if $$\Ext^2(E,E(-1))=0.$$ By Serre duality, a semistable sheaf is prioritary.  For a fixed set of numerical invariants $\xi\in K(\P^2)$, the prioritary sheaves of Chern character $\xi$ form an Artin stack $\mathcal P(\xi)$, which is an open substack of the stack of coherent sheaves.  We will need several facts about prioritary sheaves.  

\begin{theorem}\label{prioritaryThm} Let $\xi$ be a Chern character such that $\mathcal P(\xi)$ is nonempty.  \begin{enumerate}\item The stack $\mathcal P(\xi)$ is irreducible.  \cite{HirschowitzLaszlo}

\item The stack of semistable sheaves $\mathcal M(\xi)$ forms an open substack of $\mathcal P(\xi)$, which is irreducible when it is nonempty.

\item If the rank of $\xi$ is at least $2$, then the general member of $\mathcal P(\xi)$ is locally free.  \cite{HirschowitzLaszlo}

\item If the rank of $\xi$ is at least $2$, then the general member of $\mathcal P(\xi)$ is nonspecial.  That is, if $E\in \cP(\xi)$ is general, then at most one of $H^i(E)$ with $0\leq i \leq 2$ is nonzero. In particular, if $\chi(\xi)=0$, then the general $E\in \cP(\xi)$ is acyclic.  \cite{GottscheHirschowitz}
\end{enumerate}
\end{theorem}

\begin{remark}\label{prioritaryRmk} We explain how part (4) of the theorem follows from \cite{GottscheHirschowitz}.  In that paper, it is shown that if the general $E\in \cP(\xi)$ has $H^2(E) =0$, then the general $E$ is nonspecial.  Let $E\in \cP(\xi)$ be general of rank at least $2$, and suppose that $H^2(E)\neq 0$.  By \cite{HirschowitzLaszlo}, $E$ satisfies $\mu_{\max}(E)-\mu_{\min}(E)\leq 1$.  Since $H^2(E)\neq 0$, we must have $\mu(E)\leq -2$.  By part (3) of the theorem, $E$ is locally free, $\mu(E^*(-3))\geq -1$, and $\mu_{\min}(E^*(-3))\geq -2$.  Then $H^2(E^*(-3))=0$ and $E^*(-3)$ is general in its moduli stack of prioritary sheaves, so $E^*(-3)$ is nonspecial.  By Serre duality, $E$ is nonspecial as well. 

Note that the hypothesis that $\xi$ has rank at least $2$ is crucial in (4).  Indeed, if $p\in \P^2$, then $\cI_p(-3)$ is special, and every torsion-free sheaf with the same Chern character as $\cI_p(-3)$ is of the form $\cI_q(-3)$ for some $q\in \P^2$.
\end{remark}

It is also easy to increase the discriminant and decrease the number of sections of a prioritary sheaf by the following construction.  Suppose $E$ is a torsion-free coherent sheaf with $h^0(E)>0$.  Let $E\to \OO_p$ be a general map, where $p\in \P^2$ is general.  Consider the sheaf $E'$ defined by the sequence $$0\to E'\to E \to \OO_p\to 0.$$ Since $E$ has a section, the map $H^0(E)\to H^0(\OO_p)$ is surjective.  Then we compute $$  h^0(E') = h^0(E)-1, \qquad h^1(E')=h^1(E), \qquad h^2(E')=h^2(E) $$ $$\rk(E')=\rk(E), \qquad \mu(E') = \mu(E), \qquad \Delta(E') = \Delta(E) + \frac{1}{\rk(E)}, \qquad \chi(E') = \chi(E)-1.$$ Now $E'$ is again a torsion-free sheaf, so we may repeat this process so long as $E'$ has a section.  Furthermore, 
 if $E$ is prioritary, then $E'$ is prioritary.  We summarize this discussion as the following lemma.

\begin{lemma}\label{incDeltaLemma}
Let $E$ be a torsion-free coherent sheaf, let $W\subset \P^2$ be a general zero-dimensional scheme of length $w\leq h^0(E)$, let $E\to \OO_W$ be a general map, and define $E'$ by the sequence $$0\to E'\to E \to \OO_W\to 0.$$ Then $$h^0(E')=h^0(E)-w, \qquad H^1(E')\cong H^1(E), \qquad \textrm{and} \qquad H^2(E')\cong H^2(E).$$ If $E$ is prioritary, then $E'$ is also prioritary.
\end{lemma}

We now state our most general result on interpolation.  If one knows there are interpolating bundles $E$ for a scheme $Z$ for a particular slope, you can increase the slope while keeping the interpolation property. 

\begin{theorem}\label{bumpingUpThm}
Let $Z\subset \P^2$ be a nonempty zero-dimensional scheme, and suppose $E$ is a vector bundle of slope $\mu$ with $H^1(E\te \cI_Z)=0$.   For each $\mu'\geq \mu$, there is a prioritary bundle of slope $\mu'$ with interpolation for $Z$.

To be more precise, for any slope $\mu'\geq \mu$, let $\Delta'$ be the unique discriminant such that any bundle $F$ with invariants $\xi = (r',\mu',\Delta')$ has $\chi(F \te \cI_Z)=0$.  If $r'$ is sufficiently large and divisible, then the general $F\in \cP(\xi)$ has interpolation for $Z$. Moreover, if the moduli space $M(\xi)$ of semistable bundles is nonempty, then the general $F\in M(\xi)$ has interpolation for $Z$ as well.
\end{theorem}
\begin{proof}
By Corollary \ref{assumeStableCor}, we may assume $E$ is a stable vector bundle.  Let $\mu'\geq \mu$, and let $k\geq 0$ be the integer such that $$\mu+k \leq \mu' < \mu+k+1.$$  If $a,b$ are chosen appropriately, then the bundle $$F = E(k)^a \oplus E(k+1)^b$$ has slope $\mu'$.  While $F$ is obviously not stable, it is nevertheless prioritary by Serre duality.  
 
We claim $H^1(F \te \cI_Z)=0$, which will follow from the claim that $H^1(E(k) \te \cI_Z) = 0$ for all $k\geq 0$.  We prove this by induction on $k$, with $k=0$ being obvious.  Consider the sequence $$0 \to E(k-1) \to E(k) \to E(k)|_L\to 0$$ where $L$ is a general line.  Tensoring by $\cI_Z$ is exact on this sequence and $H^1(E(k-1)\te \cI_Z)=0$ by induction, so it suffices to show $H^1(E(k)|_L)=0$.  From the sequence $$0\to E(k-2)\to E(k-1)\to E(k-1)|_L\to 0,$$ we see by induction that $H^1(E(k-1)|_L)=0$ since $H^2(E(k-2))=0$ as $\mu(E(k-2))\geq -2$ and $E$ is stable.  Then the sequence $$0\to E(k-1)|_L\to E(k)|_L\to E(k)|_p \to 0$$ on $L$ gives the required vanishing.

Lemma \ref{interpOnePointLem} now shows that $H^1(F) = H^2(F) = H^1(F\te \cI_Z) = H^2(F\te \cI_Z) = 0$.  Let $w = h^0(F \te \cI_Z)$, and let $F\te \cI_Z \to \OO_W$ be a general map, where $W\subset \P^2$ is a general zero-dimensional scheme of length $w$.  If $F'$ is the kernel $$0\to F'\to F\te \cI_Z \to \OO_W\to 0,$$ then $F'$ is acyclic by Lemma \ref{incDeltaLemma}. As the map $F'\to F\te \cI_Z$ is an isomorphism near $Z$, we see that $F'$ is actually of the form $F''\te \cI_Z$ for some subsheaf $F''\subset F$, and there is an exact sequence $$0\to F''\to F\to \OO_W\to 0.$$ In particular, if $F\to \OO_W$ is a general map with kernel $F''$, then $F''\te \cI_Z$ is acyclic.  Moreover, $F''$ is prioritary since $F$ is, and $\mu(F'')=\mu(F)=\mu'$.
\end{proof}

\begin{corollary}
If $Z$ is a zero-dimensional scheme, the set of slopes of bundles with interpolation for $Z$ is either a closed interval $[\mu, \infty) \cap \Q$ with $\mu$ rational or an open interval $(\mu,\infty)\cap \Q$ (potentially with $\mu$ irrational).  The discriminant of a bundle with interpolation is determined from the slope by the requirement $\chi(E\te \cI_Z)=0$. 
\end{corollary}
\begin{proof}
Everything follows immediately from the theorem except for the statement that the set is actually non-empty.  For non-emptiness, we only need to produce a sheaf $E$ with $H^1(E\te \cI_Z)=0$, and $E = \OO_{\P^2}(n)$ for large $n$ will do.
\end{proof}

\begin{remark}
In the cases we study in this paper, it will always be the case that the interval in the previous corollary takes the form $[\mu ,\infty)\cap \Q$ with $\mu$ rational.  We expect the set of slopes of bundles with interpolation for $Z$ to always be a closed interval $[\mu, \infty) \cap \Q$ with $\mu$ rational as a consequence of the fact that $\PP^{2[n]}$ is a Mori dream space. This would follow if given a linear system $|D|$ not containing $Z$ in its stable base locus, we could always find a $\Theta$-divisor in $|mD|$, for some $m>0$, not containing $Z$.  Since we do not know whether this interval is always closed, we make the following definition.
\end{remark}

\begin{definition}
The \emph{minimal interpolating slope} $\mu_{\min}^\perp(\cI_Z)$ of a zero-dimensional scheme $Z$ is the infimum of the set of slopes of bundles with interpolation for $Z$.
\end{definition}

\begin{remark}
It can in fact happen that there are $F\in \cP(\xi)$ with interpolation for $Z$ but that $\cM(\xi)$ is empty.  For example, consider the case where $Z=p$ is a point.  If $W$ is a general collection of $2$ points, then the sheaf $$F = \OO_{\P^2}^3 \oplus \cI_W(1)$$ is prioritary of slope $1/4$ and acyclic to $\cI_p$.  But $\Delta(F) = \frac{13}{32}$, while any semistable bundle of slope $1/4$ has discriminant at least $\frac{21}{32}$ by Theorem \ref{stableClassification}.  More generally, a stable bundle of slope $\mu$ with interpolation for $p$ exists if $\mu \in \{0\} \cup [\frac{1}{3},\infty)$, while prioritary bundles with interpolation for $p$ exist for all $\mu\geq 0$.  Any prioritary bundle with slope $0 \leq \mu < \frac{1}{3}$ that has interpolation for $p$ must have $\OO_{\P^2}$ as a factor in its Harder-Narasimhan filtration.
\end{remark}

With this new language, the main theorem of \cite{HuizengaPaper2} on the cone of effective divisors on $\P^{2[n]}$ can be restated as follows.

\begin{theorem}[\cite{HuizengaPaper2}]
If $Z\in \P^{2[n]}$ is a general collection of points, then $\mu_{\min}^\perp(\cI_Z)$ is the minimum nonnegative slope of a stable vector bundle $E$ with $\chi(E\te \cI_Z)=0$.  Furthermore, the cone of effective divisors of $\P^{2[n]}$ is spanned by $$\mu_{\min}^\perp(\cI_Z)H-\frac{1}{2}B \qquad \textrm{and} \qquad B.$$
\end{theorem}

We will also need a couple similar results where the role of $\cI_Z$ is replaced by a pure $1$-dimensional sheaf.  The proofs are considerably easier.

\begin{proposition}\label{rank0stable}
Let $G$ be a pure $1$-dimensional sheaf, and suppose $E$ is a vector bundle such that $E\te G$ is acyclic.  Then $E$ and $G$ are semistable.
\end{proposition}
\begin{proof}
As in the proof of Theorem \ref{semistableThm}, consider an exact sequence $$0\to F \to E \to Q \to 0.$$ By applying a general automorphism of $\P^2$, we may assume the singularities of these sheaves avoid the support of $G$.  We obtain an exact sequence $$0\to F \te G \to E \te G \to Q \te G \to 0.$$ Since $E\te G$ is acyclic and the sheaves are $1$-dimensional, we have $\chi(Q\te G) \geq 0$.  Riemann-Roch shows that this is equivalent to $\mu(Q)\geq \mu(E)$, so $E$ is semistable.  A similar argument shows $G$ is semistable.
\end{proof}

\begin{theorem}\label{rank0bumpUp}
Let $F$ be a semistable pure 1-dimensional sheaf, and suppose $E$ is a prioritary bundle such that $E\te F$ is acyclic.  Let $(r,\mu,\Delta)$ be the numerical invariants of $E$, and fix a rational number $\Delta'\geq \Delta$.  If $r'$ is sufficiently divisible, then $E'\te F$ is acyclic for a general prioritary bundle $E'\in \cP(r',\mu,\Delta')$.
\end{theorem}

The only reason for potentially changing the rank is to ensure that $\chi(E')= r'(P(\mu)-\Delta')$ is an integer. 

\begin{proof}
Let $r'=kr$ be a multiple of $r$ such that $r'(P(\mu)-\Delta')$ is an integer.  Replacing $E$ be $E^k$ we may as well assume $r'=r$, so $r(P(\mu)-\Delta)$ and $r(P(\mu)-\Delta')$ are both integers.  Then $$\Delta' - \Delta = \frac{w}{r}$$ for some integer $w\geq 0$.  If $E\to \OO_W$ is a general map with $W$ a general zero-dimensional scheme of length $w$ and $$0\to E'\to E \to \OO_W\to 0$$ then $E'\te F\cong E\te F$ but $\Delta(E') = \Delta(E) + \frac{w}{r} = \Delta'$.
\end{proof}

If $F$ is a pure $1$-dimensional sheaf, then any bundle $E$ with $E\te F$ acyclic has slope determined by the requirement $\chi(E\te F)=0$.  The theorem shows that the admissible discriminants of such bundles form a ray. In the cases under consideration in this paper, it will always be a closed ray of the form $[\Delta,\infty)\cap \Q$, with $\Delta$ rational.

\begin{definition} The \emph{minimal interpolating discriminant} $\Delta_{\min}^\perp(F)$ of a semistable pure 1-dimensional sheaf $F$ is the infimum of the set of discriminants of bundles $E$ with $E \te F$ acyclic.
\end{definition}

\section{Correspondence between Bridgeland walls and interpolation}\label{sec-correspondence}

In this section, we demonstrate the correspondence between the geometry of a Bridgeland wall and the numerical invariants of a vector bundle orthogonal to the objects defining the wall.  The center of the wall corresponds to the slope of the orthogonal object, while the radius corresponds to the discriminant.

\begin{proposition}\label{orthogonalInvariantsProp}
Let $\xi_1,\xi_2\in K(\P^2)$ be linearly independent Chern characters with either rank $0$ or nonnegative discriminant.  Suppose $\zeta = (r,\mu,\Delta)$ is a Chern character with $r\neq 0$ and $\Delta> -\frac{1}{8}$, and $$\chi(\zeta^\ast,\xi_1) = \chi(\zeta^\ast,\xi_2)=0.$$ Then the wall $W(\xi_1,\xi_2)$ is semicircular, with center $(s,0)$ and radius $\rho$ given by $$\mu = -s-\frac{3}{2} \qquad \textrm{and} \qquad 2\Delta = \rho^2-\frac{1}{4}.$$ 

Conversely, if the nonempty semicircular wall $W(\xi_1,\xi_2)$ has center $(s,0)$ and radius $\rho$, then up to scale there is a unique $\zeta$ as above.
\end{proposition}

We begin with an elementary lemma.

\begin{lemma}\label{realPart0lemma}
Let $\xi$ be a nonzero Chern character, and let $(s_0,t_0)$ be a point in the Bridgeland plane.  Then $$\Re(\cZ_{s_0,t_0}(\xi))=0$$ if and only if there is a semicircular potential wall for $\xi$ with center $(s_0,0)$ and radius $t_0$.\end{lemma}
\begin{proof}
We will only need the case where $\xi$ has nonzero rank, so we omit the other easy cases.

Suppose $\xi=  (r,\mu,\Delta)$ with $r\neq 0$.  
We have $$\Re \left(\cZ_{s_0,t_0}(\xi)\right) = -\frac{1}{2}r((s_0-\mu)^2-t_0^2-2\Delta),$$ so if $\Re(\cZ_{s_0,t_0}(\xi))=0,$ then $|s_0-\mu|> 2\Delta$ and $$t_0 = \sqrt{(s_0-\mu)^2 - 2\Delta}.$$  There is a unique semicircular potential wall with center $(s_0,0)$, and its radius is $t_0$ (see \S \ref{sec-prelim-Bridge}).
\end{proof}

\begin{proof}[Proof of Proposition \ref{orthogonalInvariantsProp}]
$(\Rightarrow)$ If both $\xi_i$ have rank $0$, there is no $\zeta$ which gives $\chi(\zeta^\ast,\xi_1)=\chi(\zeta^\ast,\xi_2)=0$ unless $\xi_1,\xi_2$ are linearly dependent.  We, therefore, assume $\xi_1$ has nonzero rank.  

Next, we reduce to the case where $\xi_2$ also has nonzero rank.  If say $\xi_2 = (0,\ch_1,\ch_2) = (0,c,d)$, then we must have $c \neq 0$ since $\chi(\zeta^*,\xi_2) = 0$ and $\xi_2 \neq 0$.  Consider the Chern characters $\xi_1,\xi_1+k\xi_2$, where $k$ is a large integer.  Then $W(\xi_1,\xi_2) = W(\xi_1,\xi_1+k\xi_2)$ and $\chi(\zeta^*,\xi_1+ k \xi_2)=0$.  The discriminant of $\xi_1+k\xi_2$ is also nonnegative for large $k$ since it grows like $C k^2 c^2$ for some constant $C>0$.  Therefore, we may replace $\xi_2$ by $\xi_1+k\xi_2$ and prove the result for $\xi_1,\xi_1+k\xi_2$ instead.

 Assuming the ranks of the $\xi_i$ are nonzero, write $\xi_i = (r_i,\mu_i,\Delta_i)$.  Put $$s_0 = -\mu-\frac{3}{2} \qquad\textrm{and}\qquad t_0= \sqrt{2\Delta+\frac{1}{4}}.$$ We claim $\mu_{s_0,t_0}(\xi_i)=0$.  We compute \begin{eqnarray*} \Re \left(\cZ_{s_0,t_0}(\xi_i)\right) &=& -\frac{1}{2}r_i((\mu_i-s_0)^2-t_0^2-2\Delta_i)\\ &=& -\frac{1}{2}r_i\left(\left(\mu_i+\mu+\frac{3}{2}\right)^2-\frac{1}{4}-2\Delta-2\Delta_i \right) \\
&=& -r_i (P(\mu_i+\mu)- \Delta-\Delta_i)\\
&=& -\frac{1}{r} \chi( \zeta^*,\xi_i)\\ &=& 0.\end{eqnarray*} Thus $\mu_{s_0,t_0}(\xi_i)=0$ unless perhaps $\Im(\cZ_{s_0,t_0}(\xi_i)) = 0$.  If $\Im(\cZ_{s_0,t_0}(\xi_i))=0$, then we find $s_0 = \mu_i$ and $t_0^2 = -2\Delta_i$, so $$-2\Delta_i = 2\Delta +\frac{1}{4}.$$  Since $\Delta > -\frac{1}{8}$ and $\Delta_i \geq 0$, this is a contradiction.  

We conclude $\mu_{s_0,t_0}(\xi_i)=0$, so that $(s_0,t_0)$ is a point on the wall $W(\xi_1,\xi_2)$.  By Lemma \ref{realPart0lemma}, there is a potential wall for $\xi_1$ centered at $(s_0,0)$ with radius $t_0$.  Since $\Delta_1\geq 0$, potential walls for $\xi_1$ foliate the upper half-plane, and we conclude $W(\xi_1,\xi_2)$ must be this semicircle.

$(\Leftarrow)$ As before, we may assume the $\xi_i$ both have nonzero rank.  Suppose the wall $W(\xi_1,\xi_2)$ is nonempty, with center $(s_0,0)$ and radius $t_0>0$, and consider the Chern character $$\zeta = (r,\mu,\Delta)=\left(1,-s_0-\frac{3}{2},\frac{1}{2}t_0^2 - \frac{1}{8}\right).$$ Then by our earlier calculation $$0=\Re (\cZ_{s_0,t_0}(\xi_i)) = -\chi(\zeta^\ast,\xi_i)$$ since the real part of the central charge vanishes at the top point of any potential wall.  To see that $\zeta$ is unique up to scale, view the equations $$\chi(\zeta^\ast,\xi_1) = \chi(\zeta^\ast,\xi_2)=0$$ as a system of equations in the two variables $\mu,\Delta$.  Riemann-Roch shows this is equivalent to the system \begin{eqnarray*} P(\mu+\mu_1)-\Delta_1&=& \Delta \\ P(\mu+\mu_2)-\Delta_2&=& \Delta.\end{eqnarray*} Each equation gives a parabola in the $(\mu,\Delta)$-plane, and they are translates of one another.  The hypothesis that $W(\xi_1,\xi_2)$ is a semicircle instead of a line gives that $\mu_1 \neq \mu_2$, so they intersect in precisely one point.
\end{proof}

\section{Complete intersection schemes}\label{sec-ci}

In this section, we solve the interpolation problem for an arbitrary complete intersection scheme. The proof for monomial schemes will follow essentially the same basic outline, even if the details will be substantially more complicated. 
\begin{theorem}
Let $Z$ be a (potentially non-reduced) zero-dimensional complete intersection $Z = V(f,g)$, where $\deg f = a$ and $\deg g = b$, with $a\leq b$.  Then $$\mu_{\min}^\perp(\cI_Z) = b + \frac{a-3}{2}.$$
\end{theorem}
\begin{proof}
The ideal sheaf $\cI_Z$ has a resolution $$ 0 \to \OO_{\P^2}(-a-b)\to \OO_{\P^2}(-a)\oplus \OO_{\P^2}(-b) \to \cI_Z\to 0.$$  Equivalently, there is an exact sequence $$0 \to \OO_{\P^2}(-a)\to \cI_Z \to \OO_C(-b)\to 0.$$ We guess that this is the Bridgeland destabilizing sequence of $\cI_Z$.

We compute the wall $W(\OO_{\P^2}(-a),\cI_Z)$.  We have $$\ch(\OO_{\P^2}(-a)) = (1,-a,\frac{a^2}{2}), \qquad \ch(\cI_Z) = (1,0,-ab),$$ so the center of the wall is the point $(s,0)$ with $$s= -\frac{1}{2}a-b$$ and the radius $\rho$ satisfies $$\rho^2=\frac{1}{4}(a-2b)^2.$$ Then if $\zeta_\opt \in K(\P^2)$ is the class with $r(\zeta_\opt)\neq 0$ such that $$\chi(\zeta_\opt^*,\OO_{\P^2}(-a))= \chi(\zeta_\opt^*,\cI_Z)=0,$$ we must have $$\mu_\opt(\cI_Z) := \mu(\zeta_\opt)= -s-\frac{3}{2} = b+\frac{a-3}{2}$$ and $$\Delta_\opt(\cI_Z) := \Delta(\zeta_\opt)=\frac{1}{2}\rho^2-\frac{1}{8}=\frac{1}{8}((a-2b)^2-1).$$

We wish to show that if $r(\zeta_\opt)$ is sufficiently large and divisible, then a general $E\in \cP(\zeta_\opt)$ has interpolation for $Z$.  To do this we show that both $E(-a)$ and $E \te \OO_C(-b)$ are acyclic.  By Theorem \ref{prioritaryThm}, $E(-a)$ is acyclic since $\chi(E(-a))=0$.

To analyze $E\te \OO_C(-b)$, we solve the interpolation problem for $\OO_C(-b)$.  Here the destabilizing sequence is $$0\to \OO_{\P^2}(-b) \to \OO_C(-b) \to \OO_{\P^2}(-a-b)[1]\to 0.$$  The corresponding wall has the same center $$s=-\frac{1}{2}a-b$$ as the wall $W(\OO_{\P^2}(-a),\cI_Z)$, but its radius $\rho'$ satisfies $$(\rho')^2 = \frac{1}{4}a^2.$$ Observe that $$\rho^2-(\rho')^2 = b(b-a).$$ Thus the wall $W(\OO_{\P^2}(-b),\OO_C(-b))$ is always nested in $W(\OO_{\P^2}(-a),\cI_Z)$, but can be equal in case $a=b$.  If $\zeta_\opt'$ is a class orthogonal to $\OO_{\P^2}(-b)$ and $\OO_C(-b)$, then $$\mu_\opt(\OO_C(-b)):=\mu(\zeta_\opt')=b+\frac{a-3}{2}$$ and $$\Delta_\opt(\OO_C(-b)):=\Delta(\zeta_\opt') = \frac{1}{8}(a^2-1).$$ By Theorem \ref{prioritaryThm},  if $r(\zeta_\opt')$ is sufficiently large and divisible and $E\in \cP(\zeta_\opt')$ is general, then $E(-b)$ and $E(-a-b)$ are both acyclic.  Therefore, $E\te \OO_C(-b)$ is acyclic as well.  By Theorem \ref{rank0bumpUp}, the general $E \in \cP(\zeta_\opt)$ also has $E\te \OO_C(-b)$ acyclic since $\Delta_\opt(\cI_Z) \geq \Delta_\opt(\OO_C(-b))$.

We can easily verify that the guessed destabilizing sequences for $\cI_Z$ and $\OO_C(-b)$ are the actual destabilizing sequences.  The object $\OO_C(-b)$ is semistable along the wall $W(\OO_{\P^2}(-b),\OO_C(-b))$ since it is an extension of semistable objects of the same slopes. Furthermore, it is semistable for all points outside this wall since it is a Gieseker stable sheaf.  Hence, we deduce that $\cI_Z$ is semistable along the wall $W(\OO_{\P^2}(-a),\cI_Z)$ by the same argument.  

So far we have proved that $\mu_{\min}^\perp(\cI_Z) \leq \mu_\opt(\cI_Z)$.  For the other inequality, we produce a curve $\alpha$ in the Hilbert scheme passing through $Z$ such that $\alpha \cdot D_E=0$, where $D_E$ is the divisor corresponding to the vector bundle $E$ with invariants $\zeta_\opt$ having interpolation for $Z$.  To do this, keep the curve $C$ of degree $a$ fixed but vary the cutting curve of degree $b$ in a one-parameter family of curves having no common components with $C$.  Any scheme $Z'$ parameterized in this way also has the same destabilizing sequence $$0\to \OO_{\P^2}(-a)\to \cI_{Z'} \to \OO_C(-b)\to 0.$$ Then for all such $Z'$, we must have $E\te \cI_{Z'}$ acyclic, and thus $\alpha \cdot D_E = 0$.
\end{proof}

\section{Monomial objects}\label{sec-monomialObjects}

When studying the interpolation problem for a  monomial scheme $Z$, it is natural to decompose  $Z$ into  simpler monomial schemes.  If this decomposition is chosen correctly,  checking interpolation  for  $Z$ reduces to  checking interpolation for simpler schemes. The proof of Theorem \ref{thm-intro-interpolate} will be by induction on the complexity of $Z$. In this section, we will give the roadmap of the proof and describe the inductive process in detail. The actual verifications of the claims will take up the next several sections.  
 
Let $Z$ be a monomial scheme with ideal sheaf $\cI_Z$.  The main issue is to determine the Bridgeland destabilizing sequence $$0\to A \to \cI_Z\to B\to 0.$$  Then to solve the interpolation problem for $\cI_Z$, one solves the interpolation problem for the simpler objects $A$ and $B$.  While the objects $A$ and $B$ are not themselves monomial ideal sheaves, they are closely related to monomial schemes.  In the course of running this argument, there are only $3$ essentially different types of objects which arise as destabilizing subobjects or quotient objects.  This bounded complexity is what allows us to solve the interpolation problem for monomial schemes.  In this section, we introduce the $3$ different types of \emph{monomial objects}, and describe how they are destabilized.

\subsection{Notation}\label{notationSec}  Before describing the various types of monomial objects, we fix some notation common to each situation.  Consider a monomial scheme $Z$, and write $n$ for its degree.  Its block diagram $D=D_Z$ has $r(D)$ rows and $c(D)$ columns.  The minimal powers of $x,y$ in $\cI_Z$ are $x^{r(D)}$ and $y^{c(D)}$, respectively.  We write $L$ for the line $y=0$ and $L'$ for the line $x=0$.  We denote a fat line such as the nonreduced scheme $y^k=0$ by $kL$.

For any $k$ with $1\leq k \leq c(D)$, we define a monomial scheme $W_k$ corresponding to the ideal quotient $(\cI_Z: y^k)$ and another monomial scheme $Z_k$ by the intersection $Z_k = Z \cap kL$.  Similarly, for any $k$ with $1\leq k \leq r(D)$, we define a monomial scheme $W_k'$ corresponding to the ideal quotient $(\cI_Z:x^k)$ and another monomial scheme $Z_k' = Z \cap kL'$.  

In terms of the block diagram $D_Z$, the block diagrams $D_{W_k}$ and $D_{Z_k}$ split the diagram into the parts above and below the $k$th horizontal line, respectively (indexing rows from the bottom, starting at $0$ underneath the first row of boxes).   The block diagrams $D_{W'_k}$ and $D_{Z'_k}$ split $D_Z$ into the parts right and left of the $k$th vertical line, respectively (see Figure \ref{figure-notation4.1}).  We denote by $w_k$ and $w_k'$ the degree of $W_k$ and $W_k'$, respectively.  Then $\deg Z_k = n-w_k$ and $\deg Z_k' = n-w_k'$. Finally, we let $\ell$ (or $\ell(Z)$) be the number of ``full'' rows of length $c(D)$ in the block diagram, and we let $\ell'$ be the number of columns of length $r(D)$.

In further constructions, we will always denote ``vertical'' constructions with a prime (e.g., $Z_k'$) and leave ``horizontal'' constructions unadorned (e.g., $Z_k$).

 \begin{figure}[htbp]
\begin{center}
\input{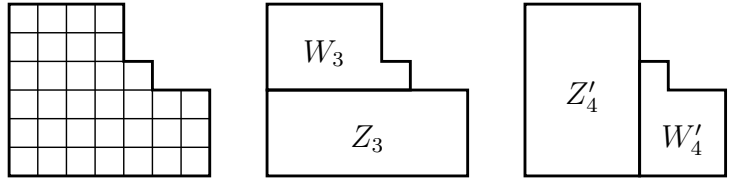}
\end{center}
\caption{The schemes $W_3, Z_3$ and $W_4'$, $Z_4'$ for a monomial scheme.}
\label{figure-notation4.1}
\end{figure}

\subsection{Rank 1 monomial objects}  Despite the complicated name, a \emph{rank 1 monomial object} is just an ideal sheaf $\cI_Z$ of a monomial scheme.  We call $\cI_Z$ \emph{trivial} if $Z$ is empty, so that $\cI_Z = \OO_{\P^2}$, and we assume $\cI_Z$ is nontrivial.  For each $k$ with $1\leq k \leq r(D)$, there is an exact sequence $$0\to \cI_{W_k}(-k) \fto{y^k} \cI_Z \to \cI_{Z_k\subset kL}\to 0.$$  For each $k$, we obtain a potential Bridgeland wall $W(\cI_{W_k}(-k),\cI_Z)$ for $\cI_Z$.  Similarly, by exchanging the roles of $x$ and $y$, we obtain another family $$ 0 \to \cI_{W'_k} (-k) \fto{x^k} \cI_Z\to \cI_{Z'_k\subset kL'}\to 0$$ of sequences which potentially destabilize $\cI_Z$.  Bridgeland walls for $\cI_Z$ are nested semicircles to the left of the vertical wall $s=0$ in the $(s,t)$-plane.  The destabilizing sequence for $\cI_Z$ always corresponds to the largest wall constructed in this way (see \S \ref{sec-nesting}).

 \begin{figure}[htbp]
\begin{center}
\input{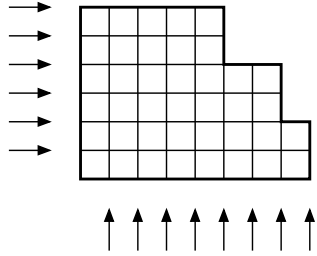}
\end{center}
\caption{The radii that need to be compared to find the destabilizing subobject of a rank 1 monomial object $\cI_Z$.}
\label{figure-rank1}
\end{figure}

Observe that (up to exchanging $x$ and $y$) the destabilizing subobject is a twist of a rank 1 monomial object $\cI_{W_k}(-k)$, with $\deg W_k < \deg Z$ (potentially $W_k$ is empty, in case $k= r(D)$).  The cokernel, however, is something new which we must study.

\subsection{Rank $0$ monomial objects} Here we describe a class of objects which (up to exchanging $x$ and $y$) contains all the destabilizing quotient objects of rank 1 monomial objects.  In order to not disrupt the flow of the argument, we will defer the simple definition of a \emph{horizontally pure} monomial scheme to \S \ref{subsec-purity}.  It combinatorially captures the condition for a rank $0$ sheaf to arise as a destabilizing quotient object of a rank $1$ monomial object.

\begin{definition}
A \emph{rank $0$ monomial object} is an ideal sheaf $\cI_{Z\subset kL}$, where $Z\subset kL$ is a horizontally pure monomial scheme whose block diagram $D$ has $r(D)=k$.
\end{definition}

For an alternate self-contained definition, we will see that $Z$ is horizontally pure if and only if $\cI_{Z\subset kL}$ is a (Gieseker) semistable sheaf (see Theorem \ref{thm-Gieseker}).

As in the rank $1$ case, we construct several potentially destabilizing subobjects of a rank $0$ monomial object, and guess that the actual destabilizing subobject corresponds to the sequence which gives the largest Bridgeland wall.  Consider the block diagram $D$ of $Z$. By assumption $r(D)=k$, so there are  $\ell'>0$ columns of $k$ boxes.  For each $i$ with $\ell' \leq i\leq c(D)$, we consider the map $$ \cI_{W_i'}(-i) \xrightarrow{x^i} \cI_{Z\subset kL}, $$ which is (typically) neither injective nor surjective as a map of sheaves. The kernel is $\OO_{\P^2}(-k-i)$ and the cokernel is $\cI_{Z_i'\subset kL \cap iL'}$.  Viewing the rows of the diagram $$\xymatrix
{
\cI_{W_i'}(-i)\ar[r] \ar[d] & \cI_{Z \subset kL} \ar[d]\\
\OO_{\P^2}(-i) \ar[r] & \cI_{Z'_i\subset kL} \\
\OO_{\P^2}(-k)\oplus \OO_{\P^2}(-i)  \ar[r]\ar[u] & \cI_{Z'_i}\ar[u]
}
$$
as complexes, the vertical maps are quasi-isomorphisms.  Let $F^\bullet$ be the complex in the final row, supported in degrees $-1$ and $0$.  Then $F^\bullet$ is the mapping cone of the map $\cI_{W_i'}(-i) \to \cI_{Z\subset kL}$, and there is a distinguished triangle $$\cI_{W_i'}(-i) \to \cI_{Z\subset kL} \to F^\bullet\to\cdot$$ 

For each $i$, we have a potential wall $$W(\cI_{W_i'}(-i),\cI_{Z\subset kL})$$ for the rank 0 object $\cI_{Z\subset kL}$.  Walls for this object are all concentric semicircles.  If we choose $i$ to maximize the radius of this wall, then in fact all the objects in the previous distinguished triangle lie in the the category $\cA_s$ along the wall, and the triangle becomes an exact sequence.  Furthermore, $\cI_{Z\subset kL}$ is first destabilized along this wall (see \S \ref{sec-nesting}).     

Again in this case, the destabilizing subobject is a twist of a rank 1 monomial object, where the monomial scheme has smaller degree than $Z$ (since $\ell'>0$ in the previous construction).  This time the rank $-1$ complex $F^\bullet$ is new.

 \begin{figure}[htbp]
\begin{center}
\input{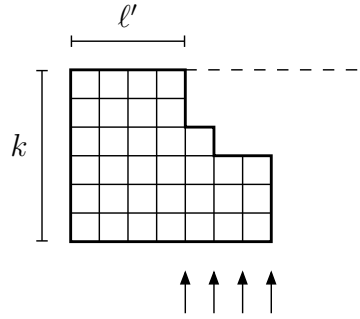}
\end{center}
\caption{The radii that need to be compared to find the destabilizing subobject of a rank 0 monomial object $\cI_{Z\subset kL}$.}
\label{figure-rank0}
\end{figure}

\subsection{Rank $-1$ monomial objects} A \emph{rank $-1$ monomial object} is a complex $F^\bullet$ of the form $$\OO_{\P^2}(-k)\oplus \OO_{\P^2}(-i) \xrightarrow{(y^k,x^i)} \cI_Z,$$ supported in degrees $-1$ and $0$, where $Z$ is a monomial scheme whose block diagram has $r(D)=k$ and $c(D)=i$. Then there are $\ell>0$ rows of length $i$ and $\ell'>0$ columns of length $k$.  The destabilizing quotient object of a rank $0$ monomial object is always a rank $-1$ monomial object (see \S \ref{sec-nesting}).

The cohomology sheaves of such a complex are given by $$\rH^{-1}(F^\bullet)= \OO_{\P^2}(-k-i), \qquad \rH^0(F^\bullet) = \cI_{Z\subset kL \cap iL'}.$$ Thus, $F^\bullet$ is in the category $\cA_s$ if and only if $s$ lies to the \emph{right} of the vertical wall $s = -i-k$.  We will show that $F^\bullet$ is Gieseker stable (see Theorem \ref{thm-Gieseker}), in the sense that it is $(s,t)$-Bridgeland stable for sufficiently large $t$.

Observe that if $Z$ is the complete intersection $kL \cap iL'$, then $F^\bullet$ is quasi-isomorphic to $\OO_{\P^2}(-i-k)[1]$; in this case we call $F^\bullet$ \emph{trivial}.  Thus, we assume that $Z$ is nontrivial. 
Potential destabilizing objects of $F^\bullet$ are analogous to the destabilizing objects in the rank 1 case.  For any $j$ with $\ell \leq j<  k$ (note the strict inequality) there are maps of complexes $$
\xymatrix{
\OO_{\P^2}(-k) \ar[r]\ar[d] & \OO_{\P^2}(-k)\oplus \OO_{\P^2}(-i) \ar[r]\ar[d] & \OO_{\P^2}(-i) \ar[d]\\
\cI_{W_j}(-j) \ar[r] & \cI_Z  \ar[r] & \cI_{Z_j\subset jL}
}
$$
giving a distinguished triangle in the derived category.  For each $j$, we obtain  a wall where the three vertical complexes here have the same Bridgeland slope.  By interchanging the roles of $x$ and $y$, we get a similar diagram and another set of walls.  Then $F^\bullet$ is destabilized along the largest wall constructed in this fashion (see \S \ref{sec-nesting}).

 \begin{figure}[htbp]
\begin{center}
\input{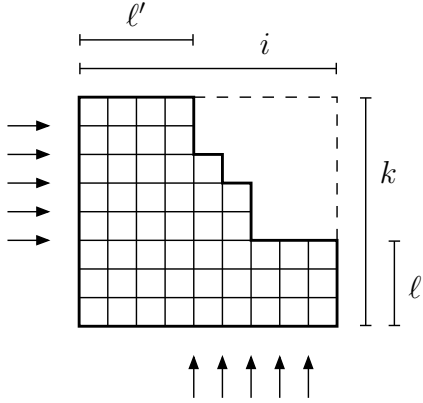}
\end{center}
\caption{The radii that need to be compared to find the destabilizing subobject of a rank $-1$ monomial object $\OO_{\P^2}(-k)\oplus \OO_{\P^2}(-i) \to \cI_Z$.}
\label{figure-rankminus1}
\end{figure}

Without loss of generality assume the largest wall corresponds to the first construction above, where $x$ and $y$ have not been swapped, and let the largest wall correspond to the index $j$, as in the diagram.  Then the destabilizing subobject is quasi-isomorphic to $\cI_{W\subset (k-j)L}(-j)$, and by our choice of $j$  the scheme $W_j$ has the required purity property for this sheaf to be a twisted rank $0$ monomial object (see Proposition \ref{prop-revisit}).  On the other hand, the quotient object $\OO_{\P^2}(-i)\to \cI_{Z_j\subset jL}$ is quasi-isomorphic to the rank $-1$ monomial object $$\OO_{\P^2}(-j)\oplus \OO_{\P^2}(-i) \to \cI_{Z_j}.$$ Since $j<k$, both schemes $W_j$ and $Z_j$ are nonempty, and hence of smaller degree than $Z$.

\begin{example}\label{bigexample}
In Figure \ref{figure-bigexample}, we illustrate the inductive procedure by decomposing the monomial scheme $Z$ of degree $48$ defined by the monomials $x^9,x^7y^2,x^6y^4,x^4y^5,x^3y^6,y^8$. In the figure, block diagrams with no dotted lines denote rank $1$ monomial objects. Rank $0$ monomial objects are indicated by a single dotted line and rank $-1$ monomial objects are denoted with two dotted lines. The left branch under any object denotes the destabilizing subobject and the right branch denotes the quotient object. A small arrow indicates the horizontal or vertical line that corresponds to the largest Bridgeland wall.

 \begin{figure}[htbp]
\begin{center}
\input{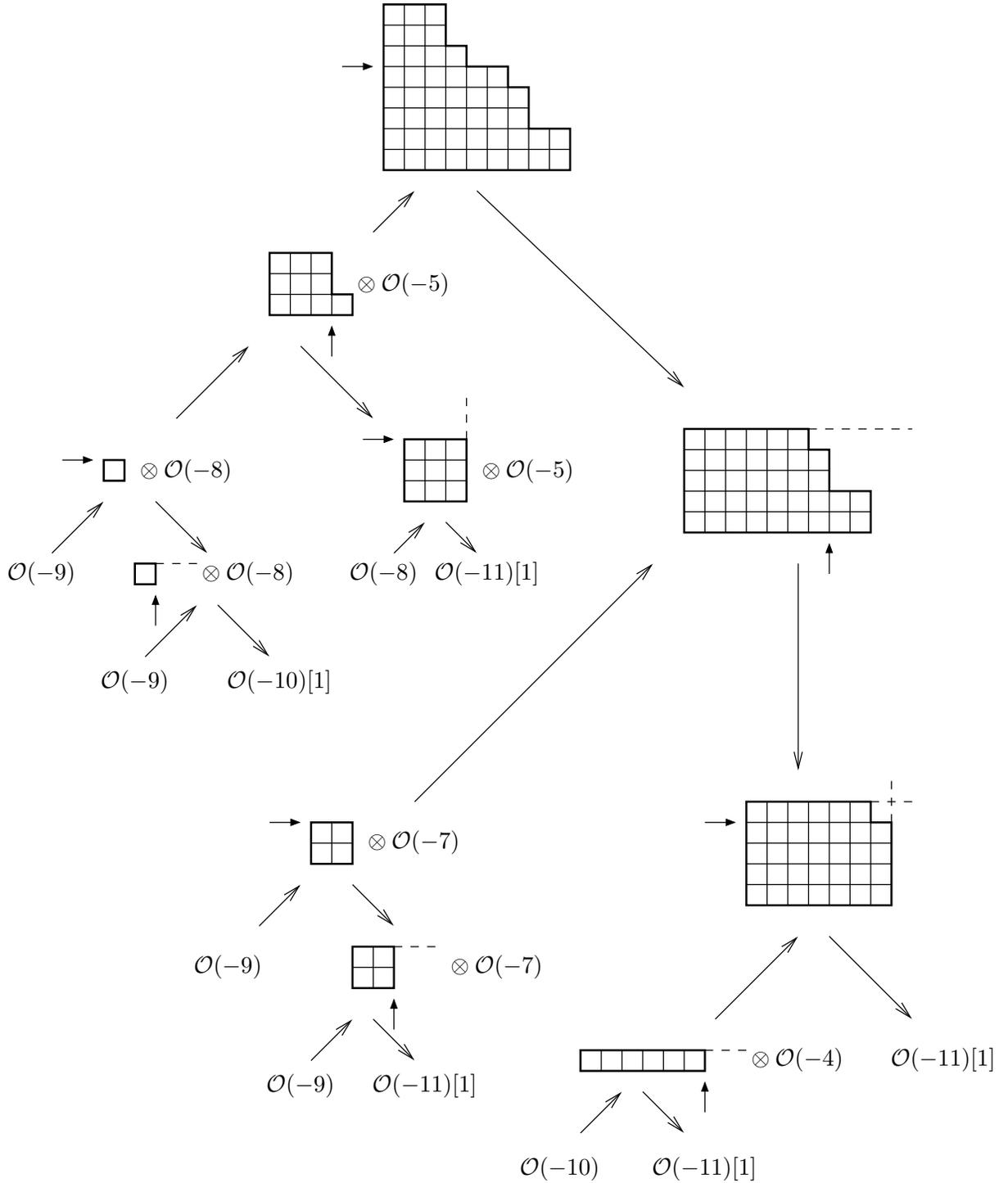}
\end{center}
\caption{Inductive procedure for decomposing the degree $48$ monomial scheme $Z$ with ideal $(x^9,x^7y^2,x^6y^4,x^4y^5,x^3y^6,y^8)$.  We have $\mu_{\min}^\perp(\cI_Z) = 8 \frac{3}{5}$.  See Example \ref{bigexample} for details.}
\label{figure-bigexample}
\end{figure}

The destabilizing sequence for $\cI_Z$ is $$0 \rightarrow \cI_{W_5}(-5) \rightarrow \cI_Z \rightarrow \cI_{Z_5 \subset 5 L} \rightarrow 0,$$ where $W_5$ is the monomial scheme with ideal $(x^4, x^3y, y^3)$ and $Z_5$ is the monomial scheme with ideal $(x^9, x^7y^2, x^6y^4, y^5)$. The left branch in the figure further decomposes $\cI_{W_5}(-5)$ until every object is a line bundle or a shift of a line bundle. For example, at the first step, $\cI_{W_5}(-5)$ has destabilizing sequence $$0 \rightarrow \cI_p(-8) \rightarrow \cI_{W_5}(-5) \rightarrow \cI_{U \subset 3 L'} (-5) \rightarrow 0,$$ where $p$ is the  point with ideal $(x,y)$ and $U$ is the monomial scheme with ideal $(x^3,y^3)$. At this point, it is easy to decompose these monomial objects into line bundles and their shifts and Figure \ref{figure-bigexample} depicts the outcome.

The right branch further decomposes the rank $0$ monomial object $\cI_{Z_5 \subset 5 L}$.  In order to inductively analyze $Z_5$ with our usually notation, set $Z=Z_5$.  Let $W_7'$ be the monomial scheme with ideal $(x^2, y^2)$ and let $Z_7'$ be the monomial scheme with ideal $(x^7, x^6y^4, y^5)$. The destabilizing sequence is $$0 \rightarrow \cI_{W_7'}(-7) \rightarrow \cI_{Z=Z_5 \subset 5 L} \rightarrow F^{\bullet} \rightarrow 0,$$ where $F^{\bullet}$ is the rank $-1$ monomial object 
$$\OO_{\PP^2}(-5) \oplus \OO_{\PP^2}(-7) \xrightarrow{(y^5, x^7)} \cI_{Z_7'}.$$ The rest of the diagram shows the further decomposition of these objects into line bundles and their shifts. 
\end{example}

\subsection{Outline of the proof of Theorems \ref{thm-intro-interpolate} and  \ref{mainThm}} Here we discuss the general outline of the structure of the rest of the paper.  The basic strategy is reminiscent of the argument for complete intersections in \S\ref{sec-ci}.

\begin{enumerate}

\item Calculate the numerical invariants associated to monomial objects.  Determine the radii of the potential Bridgeland walls, so that the actual destabilizing wall can be computed.  See \S\ref{sec-invariants}.

\item Define the horizontal purity condition in the definition of a rank $0$ monomial object (\S \ref{subsec-purity}).  Verify that the destabilizing quotient object of a rank $1$ monomial object is a rank $0$ monomial object (Proposition \ref{rank1quotientPure}), and that the destabilizing subobject of a rank $-1$ monomial object is a rank $0$ monomial object (Proposition \ref{prop-revisit}).

\item Show that each type of monomial object is Gieseker semistable (Theorem \ref{thm-Gieseker}).  
Any Gieseker semistable object is $\cZ_{s,t}$-semistable everywhere outside a single semicircular wall where it is destabilized.

\item For each monomial object $F$, consider the hypothesized destabilizing sequence $$0 \to A \to F \to B\to 0$$ which was guessed in this section.  Use the Gieseker semistability property and induction to show that $A$ and $B$ are semistable outside semicircular walls  nested inside the wall $W(A,F)$ (see \S \ref{sec-nesting}).  Conclude that this sequence is actually the destabilizing sequence for $F$.

\item Inductively solve the interpolation problem for all three types of monomial objects simultaneously by making use of the destabilizing sequences  (see \S \ref{sec-interpolate}).
\end{enumerate}

\section{Calculation of numerical invariants \& purity}\label{sec-invariants}

In this section, we collect several numerical calculations concerning Bridgeland walls associated to monomial objects defined in \S \ref{sec-monomialObjects}.  We also introduce the purity condition for a rank $0$ monomial object.  We preserve the notation from \S\ref{notationSec}.  

\subsection{Invariants of rank 1 monomial objects}\label{rank1invs}

Recall that a rank one object is just a monomial ideal sheaf $\cI_Z$.  The potential destabilizing sequences are of the form 
$$
0 \to \cI_{W_k}(-k)  \to \cI_Z\to \cI_{Z_k\subset kL}\to 0
$$
$$
0 \to \cI_{W_k'}(-k)\to \cI_Z\to \cI_{Z_k'\subset kL'}\to 0
$$
and the destabilizing sequence is the sequence above which gives the largest wall.  We compute
$$\ch(\cI_{W_k}(-k)) = (1,-k,\frac{k^2}{2}-w_k), \qquad \ch(\cI_Z)=(1,0,-n),$$ so the wall $W(\cI_{W_k}(-k),\cI_Z)$ is centered at the point $(s_k,0)$ in the Bridgeland plane with $$s_k = -\frac{k}{2}+\frac{w_k-n}{k};$$ similarly, the wall $W(\cI_{W_k'}(-k),\cI_Z)$ is centered at the point $(s_k',0)$ with $$s_k'= -\frac{k}{2}+\frac{w_k'-n}{k}.$$ Note that $\Delta(\cI_Z)=n \geq 0$, so the discussion of potential Bridgeland walls in \S \ref{sec-prelim-Bridge} applies, as does Proposition \ref{orthogonalInvariantsProp}. Potential walls for $\cI_Z$ are nested semi-circles to the left of the vertical wall $s=0$, with larger walls having more negative centers (see \cite{ABCH} and \S \ref{sec-prelim-Bridge}).  Therefore, the largest wall corresponds to the minimum number among the $s_k$ (with $1\leq k\leq r(D)$) and the $s_k'$ (with $1\leq k' \leq c(D)$).  

\begin{remark}
When $s_k^2 < 2n$, the wall constructed above is empty since the virtual radius $\rho_k$ given by $\rho_k^2 = s_k^2-2n$ is imaginary.  In \S \ref{sec-nesting}, we will show that if $s_k$ or $s_k'$ corresponds to the largest possible wall, then it is nonempty and the destabilizing sequence is an exact sequence in categories $\cA_s$ along the wall.
\end{remark}

For each $k$, let $\zeta_k\in K(\P^2)$ be a class with $r(\zeta_k)\neq 0$ such that $$\chi(\zeta_k^\ast,\cI_{W_k}(-k)) = \chi(\zeta_k^*,\cI_Z) = 0,$$ and analogously let $\zeta_k'$ be an orthogonal class for the other type of potential destabilizing sequence; these are unique up to scale.  Keeping Proposition \ref{orthogonalInvariantsProp} in mind, we define $$\mu_k(\cI_Z) := \mu(\zeta_k)=-s_k-\frac{3}{2}$$ and $$\Delta_k(\cI_Z):=\Delta(\zeta_k)=\frac{1}{2}\rho_k^2-\frac{1}{8}.$$ The slopes $\mu_k(\cI_Z)$ are precisely the slopes defined in the introduction, so they have a convenient combinatorial description in terms of block diagrams.  We similarly define $$\mu_k'(\cI_Z) = \mu(\zeta_k') \qquad \Delta_k' = \Delta(\zeta_k').$$ Then the potential wall with largest radius corresponds to a class $\zeta_\opt$ with invariants $(\mu_\opt(\cI_Z),\Delta_\opt(\cI_Z))$ given by \begin{eqnarray*} \mu_\opt(\cI_Z) &=& \max_k \{ \mu_k(\cI_Z),\mu_k'(\cI_Z)\}\\ \Delta_\opt(\cI_Z) &=& \max_k \{ \Delta_k(\cI_Z),\Delta'_k(\cI_Z)\}
\end{eqnarray*} 
In fact, $\Delta_\opt(\cI_Z)$ may be easily determined from $\mu_\opt(\cI_Z)$ by applying Riemann-Roch to the equality $\chi(\zeta_\opt^*,\cI_Z) =0$: $$\Delta_\opt(\cI_Z) = P(\mu_\opt(\cI_Z))-n.$$

\subsection{Purity of monomial schemes} \label{subsec-purity} We now introduce the purity condition in the definition of a rank 0 monomial object.  Suppose $Z$ is a monomial scheme with $k$ rows in its block diagram, and consider the $k$ horizontal slopes $\mu_1(\cI_Z),\ldots, \mu_k(\cI_Z)$ defined in the previous subsection.  We say that $Z$ is \emph{horizontally pure} if $\mu_i(\cI_Z) \leq \mu_k(\cI_Z)$ for $1\leq i\leq k$.  We will see in the next section that this condition can naturally be interpreted in terms of the Gieseker semistability of the pure $1$-dimensional sheaf $\cI_{Z\subset kL }$.  

\begin{proposition}\label{rank1quotientPure}
Let $Z$ be a monomial scheme, and suppose the largest potential wall constructed for the rank one monomial object $\cI_Z$ corresponds to the horizontal slope $\mu_\opt(\cI_Z) = \mu_k(\cI_Z)$.  Then $Z_k$ is horizontally pure, so the quotient object $\cI_{Z_k\subset kL}$ is a rank $0$ monomial object.
\end{proposition}

If the destabilizing sequence of $\cI_Z$ corresponds to splitting up $Z$ as $W_k'$ and $Z_k'$, we could discuss an analogous notion of \emph{vertically pure} monomial schemes, but we prefer to reduce to the horizontal case by swapping the $x$ and $y$ coordinates.

\begin{proof}
From the block diagram interpretation of horizontal slopes $\mu_i(\cI_Z)$, it is clear that $$\mu_i(\cI_{Z_k}) = \mu_i(\cI_Z) \leq \mu_\opt(\cI_Z)  = \mu_k(\cI_Z) =\mu_k(\cI_{Z_k})$$ for $1\leq i\leq k = r(D_{Z_k})$. 
\end{proof}

\subsection{Invariants of rank $0$ monomial objects}\label{ss-rankZeroInv} Consider a rank $0$ monomial object $F = \cI_{Z\subset kL}$, where $Z$ is a horizontally pure scheme with $r(D)=k$.  To construct potential walls for this object, we let $\ell'$ be the number of columns in $D$ of length $k$ and allowed $i$ to be any index with $\ell'\leq i \leq c(D)$.  We then considered a sequence $$ 0\to \cI_{W'_i}(-i)\to F \to (\OO_{\P^2}(-k)\oplus \OO_{\P^2}(-i)\to \cI_{Z_i'})\to 0.$$ We compute Chern characters $$\ch(\cI_{W_i'}(-i)) = (1,-i,\frac{i^2}{2}-w_i') \qquad \ch( F) = (0,k,-\frac{k^2}{2}-n).$$ All Bridgeland walls for this object are concentric circles with fixed center $$s_0 = -\frac{n}{k}-\frac{k}{2},$$ and the radius of the wall corresponding to this sequence is $\rho_i'$ given by $$(\rho_i')^2 = s_0^2- \left(i(k-i)+\frac{2ni}{k}+w_i'\right).$$

As in the rank $1$ case, let $\zeta_i' \in K(\P^2)$ be a class with nonzero rank such that $\chi((\zeta_i')^\ast,\xi)=0$ for each Chern character $\xi$ of a term in the destabilizing sequence.  We define \begin{eqnarray*} \Delta_i'(F)  = \Delta(\zeta_i') = \frac{1}{2}(\rho_i')^2-\frac{1}{8}\end{eqnarray*}
and
\begin{eqnarray*}
\mu_\opt(F) &=& -s_0-\frac{3}{2}\\
\Delta_\opt(F)&=& \max_i \{\Delta_i'(F)\}.
\end{eqnarray*}
Hence, $\mu_\opt(F)$ depends only on $\ch(F)$ and not on the finer structure of $Z$ itself, while maximizing the radius of the wall corresponds to maximizing the $\Delta_i'(F)$.

\subsection{Invariants of rank $-1$ monomial objects} Here the story is similar to rank $1$ monomial objects, so we will be more brief.  Let $F^\bullet = \OO_{\P^2}(-k)\oplus \OO_{\P^2}(-i)\to \cI_Z$.  The potential ``horizontal'' destabilizing sequences look like $$0\to \cI_{W_j\subset (k-j)L}(-j) \to F^\bullet \to (\OO_{\P^2}(-j)\oplus \OO_{\P^2}(-i)\to \cI_{Z_j})\to 0$$ with $\ell \leq j < k$, and there are analogous potential ``vertical'' destabilizing sequences.  The Chern characters of two objects from the sequence are given by $$\ch(F^\bullet)=(-1,k+i,-\frac{k^2+i^2}{2}-n) \qquad \ch(\OO_{\P^2}(-j)\oplus\OO_{\P^2}(-i)\to \cI_{Z_j})=(-1,j+i,-\frac{j^2+i^2}{2}-(n-w_j)),$$ from which we compute the center $(s_j,0)$ of the corresponding Bridgeland wall as $$s_j = -\frac{1}{2}(j+k) - \frac{w_j}{k-j}.$$ Analogously there are centers of the form $s_j'$ coming from potential vertical destabilizing sequences.  We have $\Delta(F^\bullet) = ki-n\geq 0$, so the discussion of potential Bridgeland walls from \S \ref{sec-prelim-Bridge} applies.  Potential walls for $F^\bullet$ lie to the right of the vertical wall $s = -k-i$, and grow larger as their centers increase.  Maximizing the Bridgeland wall then corresponds to maximizing the $s_j$ and $s_j'$.  Denote by $\rho_j,\rho_j'$ the virtual radii of the corresponding walls.

We choose classes $\zeta_j,\zeta_j'\in K(\P^2)$ orthogonal to the terms of the corresponding destabilizing sequences.  Define $$\begin{array}{rcccl} \mu_j(F^\bullet) &=& \mu(\zeta_j) &=&  \displaystyle -s_j-\frac{3}{2}\\[1em]
\Delta_j(F^\bullet) &=& \Delta(\zeta_j) &=& \displaystyle\frac{1}{2} \rho_j^2-\frac{1}{8},
\end{array}
$$
and analogously define $\mu_j'(F^\bullet),\Delta_j'(F^\bullet)$.  Then \begin{eqnarray*}
\mu_\opt(F^\bullet) &=& \min_j \{\mu_j(F^\bullet),\mu_j'(F^\bullet)\}\\
\Delta_\opt(F^\bullet) &=& \max_j \{\Delta_j(F^\bullet), \Delta_j'(F^\bullet)\},
\end{eqnarray*}
noticing that by contrast with the rank 1 case we must \emph{minimize} the slopes.

\subsection{Purity revisited}
As a last result in this section, let us verify that the rank $0$ destabilizing subobject of a rank $1$ monomial object satisfies the necessary purity condition.  We preserve the notation from the previous subsection.

\begin{proposition}\label{prop-revisit}
Suppose the largest potential destabilizing wall for the rank $-1$ monomial object $F^\bullet$ corresponds to a sequence $$0\to \cI_{W_j\subset (k-j)L}(-j) \to F^\bullet \to (\OO_{\P^2}(-j)\oplus \OO_{\P^2}(-i)\to \cI_{Z_j})\to 0.$$ Then $W_j$ is a horizontally pure monomial scheme.  The destabilizing subobject of a rank $-1$ monomial object is a twisted rank $0$ monomial object.
\end{proposition}
\begin{proof}
The hypothesis gives $\mu_\opt(F^\bullet) = \mu_j(F^\bullet)$.  The block diagram $D_{W_j}$ has $r(D_{W_j})= k-j$ since $r(D_Z)=k$.
We must check that $\mu_{k-j}(\cI_{W_j})\geq \mu_{m}(\cI_{W_j})$ for each $1\leq m < k-j$.  The part of the block diagram of $W_j$ lying above the $m$th horizontal just corresponds to the scheme $W_{j+m}$.  This inequality therefore amounts to showing \begin{equation}\label{purityIneq} \frac{w_j}{k-j} + \frac{k-j-3}{2} \geq \frac{w_j - w_{j+m}}{m} + \frac{m-3}{2}.\end{equation}  Since $\mu_\opt(F^\bullet)= \mu_j(F^\bullet)$, we are given the inequality $\mu_j(F^\bullet) \leq \mu_{j+m}(F^\bullet),$ or $$\frac{1}{2}(j+k) + \frac{w_j}{k-j}\leq \frac{1}{2}(j+m+k)+\frac{w_{j+m}}{k-j-m}.$$ Rearrange this to get $$w_{j+m} \geq (k-j-m)\left(\frac{w_j}{k-j}-\frac{1}{2}m\right).$$ Substituting this estimate into Inequality \ref{purityIneq} yields an equality, so the required inequality is true.
\end{proof}

\section{Gieseker semistability of monomial objects}\label{sec-Gieseker}

We say that an object $F$ of a category $\cA_s$ is \emph{Gieseker semistable} if it is $\cZ_{s,t}$-semistable for all sufficiently large $t$.  This notion is independent of the choice of $s$.  If $F$ is a pure sheaf, then this notion of Gieseker semistability coincides with the usual notion of Gieseker semistability of sheaves.

The goal of this section is to prove the following result.

\begin{theorem}\label{thm-Gieseker}
Monomial objects are Gieseker semistable.
\end{theorem}

The Gieseker semistability condition ensures that a monomial object which is destabilized along a semicircular wall is $\cZ_{s,t}$-semistable at all points lying outside the semicircle \cite{ABCH}.  In particular, if another semicircular potential wall lies outside the wall where a monomial object is destabilized, then that object is semistable along the larger wall.

As usual, there are separate arguments to be given depending on the rank of the monomial object.  The rank $1$ case is clear, since ideal sheaves are Gieseker stable.

\begin{proposition}
Rank $0$ monomial objects are Gieseker semistable sheaves.
\end{proposition}
\begin{proof}
Consider a rank $0$ monomial object $\cI_{Z\subset kL}$.
Any subsheaf of $\cI_{Z\subset kL}$ is itself an ideal sheaf, so corresponds to a closed subscheme $Z'\subset kL$.  Any homogeneous ideal $$J \subset \C [x,y,z]/(y^k)$$ is of the form $y^j\cdot J'$ for some homogeneous ideal $J'$ such that the subscheme of $(k-j)L$ defined by $J'$ is zero-dimensional.  Denoting by $\tilde J$ the ideal sheaf on $kL$ corresponding to the ideal $J$, there is an inclusion $\tilde J \subset \cI_{Z\subset kL}$ if and only if the subscheme defined by $J'$ contains the scheme $W_j$ constructed from $Z$.  Then if  $\tilde J \subset \cI_{Z\subset kL}$, we conclude that the reduced Hilbert polynomial satisfies $p(\tilde J) \leq p( \cI_{Z \cup jL \subset kL}).$  Thus it suffices to check that $p(\cI_{Z\cup jL \subset kL}) \leq p(\cI_{Z\subset kL})$ for $0 < j < k$.

We have an exact sequence $$ 0 \to \cI_{Z\cup jL \subset kL} \to \cI_{Z \subset kL} \to \cI_{Z_j \subset jL} \to 0. $$ It is enough to show that $p(\cI_{Z\subset kL}) \leq p(\cI_{Z_j\subset jL})$.  Using the sequences $$\begin{array}{ccccccccc} 0 &\to &\cI_{Z\subset kL} &\to &\OO_{kL} &\to &\OO_Z &\to & 0 \\ 0& \to & \OO_{\P^2}(-k) &\to &\OO_{\P^2} &\to &\OO_{kL} &\to & 0\end{array}$$
we find that the (ordinary) Hilbert polynomial of $\cI_{Z\subset kL}$ is $$P_{\OO_{kL}}(x) - P_{\OO_Z}(x)=P_{\OO_{\P^2}}(x) - P_{\OO_{\P^2}(-k)}(x)-n= kx - \left(n+\frac{k^2-3}{2}\right),$$ so the reduced Hilbert polynomial of $\cI_{Z\subset kL}$ is $$p(\cI_{Z\subset kL}) = x-\mu_k(\cI_Z).$$
Likewise, we find $$p(\cI_{Z_j\subset jL}) = x- \mu_j(\cI_{Z_j}) = x - \mu_j(\cI_Z).$$ The required inequality then follows from the horizontal purity of $Z$.
\end{proof}

\begin{proposition}
Rank $-1$ monomial objects are Giesker stable.
\end{proposition}
\begin{proof}
Let $$F^\bullet := (\OO_{\P^2}(-k) \oplus \OO_{\P^2}(-i) \to \cI_Z)$$ be a rank $-1$ monomial object, and let $s> - i-k$, so that $F^\bullet\in \cA_s$.  If $F^\bullet$ is not $\cZ_{s,t}$-stable for all large $t$, there must be a destabilizing subobject $A^\bullet \subset F^\bullet $ such that $\mu_{s,t}(A^\bullet) > \mu_{s,t}(F^\bullet)$ for all large $t$.  Consider the exact sequence in $\cA_s$ $$0\to A^\bullet \to F^\bullet \to B^\bullet \to 0,$$ which induces an exact sequence of cohomology sheaves $$0\to \rH^{-1}(A^\bullet) \fto\alpha \OO_{\P^2}(-k-i) \to \rH^{-1}(B^\bullet) \to \rH^0(A^\bullet) \to \cI_{Z\subset kL \cap iL'} \to \rH^0(B^\bullet)\to 0.$$ Then $\rH^{-1}(A^\bullet)$ is a subsheaf of $\OO_{\P^2}(-k-i)$, and the cokernel of $\alpha$ is a subsheaf of $\rH^{-1}(B^\bullet)$.  Since $\rH^{-1}(B^\bullet)$ is torsion-free (as it is in $\cF_s$), the cokernel of $\alpha$ must be torsion-free.  Hence, either $\rH^{-1}(A^\bullet) = \OO_{\P^2}(-k-i)$ and $\alpha$ is an isomorphism; or $\rH^{-1}(A^\bullet) = 0$.

First, suppose $\rH^{-1}(A^\bullet) = 0$.  Since the rank of $F^\bullet$ is $-1$, we have $\mu_{s,t} (F^\bullet) \to \infty$ as $t\to \infty$.  On the other hand, as $t\to \infty$ we have either
\begin{enumerate}
\item $\mu_{s,t}(A^\bullet) \to -\infty$ if $\rH^0(A^{\bullet})$ has positive rank,
\item $\mu_{s,t}(A^\bullet) \to 0$ if $\rH^0(A^\bullet)$ is $1$-dimensional, or
\item $\mu_{s,t}(A^\bullet) \equiv \infty$ if $\rH^0(A^\bullet)$ is $0$-dimensional.
\end{enumerate} 
Thus, in the first two cases $A^\bullet$ cannot destabilize $F^\bullet$ for large $t$.  The final case also cannot happen because there are no nonzero homomorphisms from a $0$-dimensional sheaf to $F^\bullet$.  To see this, it suffices to show that $\Hom_{D^b(\P^2)}(\OO_p,F^\bullet)=0$ for any point $p$.  There is a distinguished triangle $$\cI_Z \to F^\bullet \to (\OO_{\P^2}(-i)\oplus \OO_{\P^2}(-k))[1]\to \cdot,$$ and applying $\Hom_{D^b(\P^2)}(\OO_p,-)$ shows the required vanishing.

Next, suppose $\rH^{-1}(A^\bullet) = \OO_{\P^2}(-k-i)$ and $\alpha$ is an isomorphism.  The previous long exact sequence of cohomology sheaves induces a short exact sequence $$0\to \rH^{-1}(B^\bullet) \to \rH^0(A^\bullet) \to T\to 0,$$ where $T \subset \cI_{Z\subset kL\cap iL'}$ is a $0$-dimensional sheaf.  Then $\rH^{-1}(B^\bullet)$ and $\rH^0(A^\bullet)$ have the same rank $r$ and first Chern class $c_1$.  If $r \neq 0$, these sheaves have the same slope, which is absurd since $\rH^{-1}(B^\bullet)\in \cF_s$ and $\rH^0(A^\bullet) \in \cQ_s$.  Thus $r=0$, and since $\rH^{-1}(B^\bullet)$ is torsion-free, we find $\rH^{-1}(B^\bullet)=0$.  Then $B^\bullet$ is a (nonzero) $0$-dimensional sheaf, so $\mu_{s,t}(B^\bullet) \equiv \infty$.  This contradicts the assumption that $B^\bullet$ was a destabilizing quotient object of $F^\bullet$.
\end{proof}

\section{Nesting of Bridgeland walls for monomial objects}\label{sec-nesting}

Int this section, we show that the conjectured destabilizing walls for monomial objects are the actual destabilizing walls.  Recall that  a monomial object is \emph{trivial} if it is either a line bundle (in the rank $1$ case) or a shift of a line bundle (in the rank $-1$ case).  For simplicity, we will also call twisted trivial monomial objects trivial.  Rank $0$ monomial objects are never trivial in this sense.  Trivial monomial objects are always $\cZ_{s,t}$-stable if they are in $\cA_s$ \cite{ABCH}.

\begin{theorem}
Let $F$ be a nontrivial monomial object, and consider the largest potential destabilizing wall for $F$, as constructed in \S \ref{sec-monomialObjects}.  Then $F$ is destabilized along this wall, and $F$ is $\cZ_{s,t}$-semistable at all points on and outside this wall.

Precisely, consider the distinguished triangle $$A\to F\to B\to \cdot.$$ corresponding to the wall.  Then the wall $W(A,F)$ is nonempty, and there is an exact sequence $$0\to A \to F \to B\to 0$$ of semistable objects of the categories $\cA_s$ along the wall.  The object $A$ is either trivial or it is a twisted monomial object destabilized along a wall nested inside $W(A,F)$.  Similarly, the monomial object $B$ is either trivial or it is a monomial object destabilized along a wall nested inside $W(A,F)$.
\end{theorem}

The nonemptiness of the wall and the exactness of the sequence along the wall are both consequences of the nesting statements (see \S \ref{ss-nonemptiness} and \S \ref{ss-exactness}).  Since monomial objects are Gieseker semistable, they are all $\cZ_{s,t}$-semistable outside their destabilizing walls.  We conclude that proving the theorem reduces to proving the nesting statements, which forms the technical heart of the argument.

For the proof, we consider each type of nontrivial monomial object $F$ separately and proceed by induction on the degree $n$ of the monomial scheme in a monomial object.   We summarize facts which were discussed in section \ref{sec-monomialObjects} to describe the structure of the induction.
\begin{enumerate} \item Suppose $F$ has rank $-1$.  Its destabilizing subobject is a twisted rank $0$ object corresponding to a monomial scheme of smaller degree, and its destabilizing quotient object is a (possibly trivial) rank $-1$ object corresponding to a monomial scheme of smaller degree.  We conclude by induction.

\item Suppose $F$ has rank $0$.  The destabilizing subobject is a (potentially trivial) twisted rank $1$ object corresponding to a monomial scheme of smaller degree.  Its quotient object is a (potentially trivial) rank $-1$ object corresponding to a monomial scheme of no larger degree.  Using induction on the subobject and applying the rank $-1$ case to the quotient object allows us to conclude.

\item Finally, suppose $F$ has rank $1$.  Similarly to the rank $0$ case,  a rank $1$ object gives another ``smaller'' (potentially trivial) twisted rank $1$ object and a rank $0$ object which corresponds to a monomial scheme of no larger degree.  We use induction on the rank $1$ object and run the preceding argument for rank $0$ objects on the quotient.
\end{enumerate}

Thus to prove the nesting statements for $F$, we may assume inductively that the theorem is known for $A$ and $B$.  Before showing the nesting statements, let us show that the nonemptiness of the wall and exactness of the sequence follow from nesting.  The same inductive structure is useful for proving these facts.

\subsection{Nonemptiness of destabilizing walls}\label{ss-nonemptiness} Assume the nesting result holds.   Observe that the only way both the subobject $A$ and the quotient object $B$ of $F$ can be trivial is if $F = \cI_{Z\subset kL}$ has rank $0$ and $Z$ is a complete intersection scheme $kL\cap iL'$, in which case the destabilizing sequence is $$0\to \OO_{\P^2} (-i)\to F\to \OO_{\P^2}(-k-i)[1]\to 0.$$ Since $k>0$, the corresponding wall is nonempty.

If either $A$ or $B$ is a nontrivial twisted monomial object, then by induction (as described above) we conclude that object is destabilized along a nonempty wall.  By nesting, $F$ is destabilized along a nonempty wall as well.  

\subsection{Exactness of destabilizing triangles}\label{ss-exactness}  Again assume the nesting result holds, and consider the triangle $$A\to F \to B \to \cdot$$ If $A$ is nontrivial, then by induction it is destabilized along a nonempty wall nested inside $W(A,F)$.  Then it is in the categories $\cA_s$ along its destabilizing wall, so it is in some of the categories $\cA_s$ along $W(A,F)$ as well.  It is then in the categories along the entire wall.  The same argument applies to $B$.

We must also see that $F$ lies in the categories $\cA_s$ along the wall $W(A,F)$.  If $F$ has rank $0$, there is nothing to prove.  In the rank $1$ case, one uses the formulas of \S\ref{sec-invariants} to check the wall $W(A,F)$ lies to the left of the vertical wall for $F$, and in the rank $-1$ case the wall $W(A,F)$ lies to the right of the vertical wall for $F$.

Finally, we must consider what happens when either $A$ or $B$ is trivial.  There are four cases to consider.

\emph{Case 1: $A$ trivial, $F$ has rank $1$}.  Say $F = \cI_Z$.  All potential walls for $F$ left of the vertical wall for $F$ are semicircles with centers left of $-\sqrt{2\Delta(\cI_Z)}=-\sqrt{2n}.$  We must see that if $A = \OO_{\P^2}(-k)$ then $-k > -\sqrt{2n}$, or $k^2 < 2n$.  Up to swapping $x$ and $y$, the only way to have $A = \OO_{\P^2}(-k)$ is if $r(D_Z)=k$.  Then $$\mu_\opt(\cI_Z) = \mu_k(\cI_Z) = \frac{n}{k}+\frac{k-3}{2}$$ while $$\mu'_1(\cI_Z) = k-1,$$ and by maximality of the wall $\mu_k(\cI_Z) \geq \mu_1'(\cI_Z).$ The required inequality follows.

\emph{Case 2: $A$ trivial, $F$ has rank $0$}.  We let $F = \cI_{Z\subset kL}$ and $A = \OO_{\P^2}(-i)$.  It is enough to see that $-i > s_0$, where $s_0$ is the center of all the potential walls for $F$.  By \S \ref{ss-rankZeroInv}, this amounts to the inequality $$i<\frac{n}{k}+\frac{k}{2}.$$ We have $i = c(D_Z)$, so horizontal purity of $Z$ gives $$i-1=\mu_1(\cI_Z) \leq \mu_k(\cI_Z) = \frac{n}{k}+\frac{k-3}{2},$$ and the inequality is immediate.

\emph{Case 3: $B$ trivial, $F$ has rank $0$}.  Again say $F = \cI_{Z \subset kL}$, and write $B = \OO_{\P^2}(-k-i)[1]$.  The block diagram $D_Z$ must then have $\ell'=i$ full columns of $k$ boxes.  We must see that $-k-i \leq s_0$, or $$ i \geq \frac{n}{k}-\frac{k}{2}.$$ Considering that $w_{k-1}=i$, this follows immediately from the purity hypothesis $\mu_{k}(\cI_Z)\geq \mu_{k-1}(\cI_Z)$.

\emph{Case 4: $B$ trivial, $F$ has rank $-1$}.  Write $$F = (\OO_{\P^2}(-k) \oplus \OO_{\P^2}(-i) \to \cI_Z) \qquad \textrm{and} \qquad B = \OO_{\P^2}(-j-i)[1], $$ so that $F$ is destabilized by the $j$th horizontal line and there are $\ell = j$ full rows of $i$ boxes in $D_Z$.  Walls for $F$ are nested semicircles to the right of the vertical wall $s = -k-i$, and they all have centers lying to the right of $$-k-i+\sqrt{2\Delta(F)} = -k-i+\sqrt{2(ki-n)}.$$ It is then enough to show $$-j-i \leq -k-i + \sqrt{2(ki-n)},$$  or $(k-j)^2 < 2(ki-n).$  This follows from  the inequality $\mu_j(F) \leq \mu_{i-1}'(F)$.

In every case, $A,$ $F,$ and $B$ are all objects in $\cA_s$ along the wall, and the triangle is an exact sequence.

\subsection{Nesting for rank $1$ objects} Consider the conjectured destabilizing sequence of a rank $1$ object $$0 \to \cI_{W_k}(-k)\to \cI_Z \to \cI_{Z_k\subset kL}\to 0.$$ By induction, we may assume both the subobject and the quotient object have the conjectured destabilizing sequences.

\begin{lemma}\label{rankOneNest1}
If $\cI_{W_k}(-k)$ is nontrivial, it is destabilized along a wall nested in $W(\cI_{W_k}(-k),\cI_Z)$, and $$\mu_\opt(\cI_{W_k})+k \leq \mu_\opt(\cI_Z).$$
\end{lemma}
\begin{proof}
The destabilizing wall for $\cI_{W_k}(-k)$ is centered $k$ units to the left of the destabilizing wall for $\cI_{W_k}$.  Since the slopes $\mu_\opt(\cI_{W_k})$ and $\mu_\opt(\cI_Z)$ correspond to the centers of the destabilizing walls for $\cI_{W_k}$ and $\cI_Z$, the nesting statement is equivalent to the displayed inequality.
 
 By assumption, $\mu_\opt(\cI_Z) = \mu_k(\cI_Z)$.  There are two cases to consider, corresponding to whether $\mu_\opt(\cI_{W_k}) = \mu_i(\cI_{W_k})$ or $\mu_\opt(\cI_{W_k}) = \mu'_i(\cI_{W_k})$.

First, suppose $\mu_\opt(\cI_{W_k})=\mu_i(\cI_{W_k})$.  By the block diagram interpretation of horizontal slopes (see Figure \ref{figure-relation1}), we find $$(i+k) \mu_{i+k}(\cI_Z)= k\mu_k(\cI_Z)+i\mu_i(\cI_{W_k})+ik.$$ The choice of $k$ gives $\mu_{i+k}(\cI_Z) \leq \mu_k(\cI_Z),$ from which we conclude $$\mu_\opt(\cI_{W_k})+ k = \mu_i(\cI_{W_k})+k \leq \mu_k(\cI_Z) = \mu_\opt(\cI_Z). $$

 \begin{figure}[htbp]
\begin{center}
\input{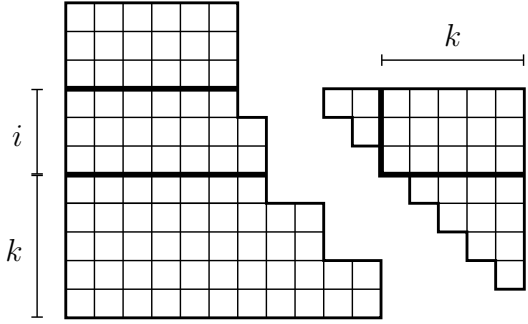}
\end{center}
\caption{The block diagram interpretation of the equation $(i+k) \mu_{i+k}(\cI_Z)= k\mu_k(\cI_Z)+i\mu_i(\cI_{W_k})+ik.$}
\label{figure-relation1}
\end{figure}

On the other hand, suppose $\mu_\opt(\cI_{W_k}) = \mu_i'(\cI_{W_k})$.  Again by the block diagram interpretation of vertical slopes (see Figure \ref{figure-relation2}), we have $$\mu_\opt(\cI_{W_k}) = \mu_i'(\cI_{W_k}) = k + \mu_i'(\cI_Z) \leq  k + \mu_\opt(\cI_Z),$$ as required.

 \begin{figure}[htbp]
\begin{center}
\input{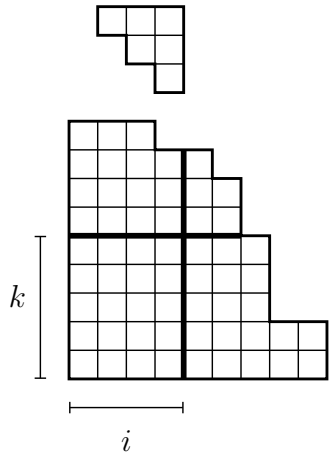}
\end{center}
\caption{The block diagram interpretation of the equation $\mu_i'(\cI_{W_k}) = k + \mu_i'(\cI_Z).$}
\label{figure-relation2}
\end{figure}

\end{proof}

\begin{lemma}\label{rankOneNest2}
The rank zero object $\cI_{Z_k\subset kL}$ is destabilized along a wall nested in $W(\cI_Z,\cI_{Z_k\subset kL}),$ and $$\Delta_\opt(\cI_{{Z_k}\subset kL}) \leq \Delta_\opt(\cI_Z).$$ (Note that both objects have the same $\mu_\opt$.)
\end{lemma}
\begin{proof}
Write $F = \cI_{Z_k\subset kL}$.  In the block diagram $D_{Z_k}$ of $Z_k$, there are some number $\ell'>0$ of columns of full height $k$.  By induction, $F$ is destabilized by some twisted rank one object $\cI_{W_i'}(-i)$ with $\ell' \leq i \leq c(D(Z_k)).$  In particular, since $i\geq \ell'$, notice that  $W_i'$ can be obtained as a vertical slice from either $Z$ or from $Z_k$.

Since $\Delta_\opt(F)$ and $\Delta_\opt(\cI_Z)$ correspond to the radii of the corresponding destabilizing walls, the nesting statement amounts to the inequality $\Delta_\opt(F) \leq \Delta_\opt(\cI_Z)$.
To show the inequality, note that by the choice of $k$, we have $$\frac{n-w_i'}{i} + \frac{i-3}{2} =\mu_i'(\cI_Z) \leq \mu_k(\cI_Z) = \mu_k(\cI_{Z_k}) = \frac{n-w_k}{k}+\frac{k-3}{2}.$$ Interpret this inequality as a lower bound for $w_i'$:  
$$w_{i}' \geq \frac{1}{2} i(i-k) +n - \frac{i}{k}(n-w_k).$$
Recalling our computations from Section \ref{sec-invariants},
$$\Delta_\opt(F) = \Delta_i(F) = P(\mu_k(\cI_Z)-i)-w_i'$$ and $$\Delta_\opt(\cI_Z) = P(\mu_k(\cI_Z))-n.$$ Then $$\Delta_\opt(\cI_Z)-\Delta_\opt(F) = -\frac 12 i(i-k)-n+\frac{i}{k}(n-w_k)+w_i'.$$ Substituting the lower bound for $w_i'$ gives $\Delta_\opt(F)-\Delta_\opt(\cI_Z) \geq 0.$
\end{proof}

\subsection{Nesting for rank $0$ objects} \label{ss-rank0nesting}
Here we consider a rank $0$ object $F = \cI_{Z\subset kL}$ and its destabilizing sequence $$0\to \cI_{W_i'}(-i) \to F \to G^\bullet \to 0,$$ where $G^\bullet$ is the rank $-1$ object $$G^\bullet:= \OO_{\P^2}(-k)\oplus \OO_{\P^2}(-i)\to \cI_{Z_i'}.$$  To show that the walls nest appropriately, the purity hypothesis on $Z$ is essential.

\begin{lemma}\label{rankZeroNest1}
If $\cI_{W_i'}(-i)$ is nontrivial, it is destabilized along a wall nested in $W(\cI_{W_i'}(-i),F)$, and $$\mu_\opt(\cI_{W_i'})+i\leq \mu_{\opt}(F).$$
\end{lemma}
\begin{proof}
As in previous cases, the inequality precisely encodes the nesting of the walls.

First, suppose $W_i'$ is destabilized horizontally, with $\mu_j(\cI_{W_i'}) = \mu_\opt(\cI_{W_i'})$.  From our explicit calculations in Section \ref{sec-invariants}, we have $\mu_\opt(F) = \mu_k(\cI_Z)$, and since $Z$ is horizontally pure $\mu_k(\cI_Z)\geq \mu_j(\cI_Z)$.  By the block diagram interpretation of horizontal slopes, $$\mu_\opt(\cI_{W_i'})+i=\mu_j(\cI_{W_i'})+i =\mu_j(\cI_Z)\leq \mu_k(\cI_Z)=\mu_\opt(F).$$

On the other hand, suppose $W_i'$ is destabilized vertically, with $\mu_j'(\cI_{W_i'})= \mu_\opt(\cI_{W_i'}).$ The required inequality is $$\frac{w_i'-w_{i+j}'}{j}+\frac{j-3}{2} +i = \mu'_j(\cI_{W_i'})+i\leq \mu_\opt(F) = \frac{n}{k}+\frac{k-3}{2}$$
Here the choice of $i$ gives an inequality $\Delta_i(F)\geq \Delta_{i+j}(F)$. Written out, this is equivalent to $$w_{i+j}'\geq \frac{1}{2}j(2i+j-k)-\frac{j}{k}n+w_i'.$$ Making this substitution proves the required inequality.
\end{proof}

\begin{lemma}\label{rankZeroNest2}
If the rank $-1$ object $G^\bullet$ is nontrivial, it is destabilized along a wall nested in $W(F,G^\bullet)$, and $$\mu_\opt(G^\bullet)\geq \mu_\opt(F).$$
\end{lemma}
\begin{proof}
Suppose $G^\bullet$ is destabilized by a rank zero object $A$.  The walls $W(G^\bullet,A)$ and $W(F,G^\bullet)$ are both walls for $G^\bullet$, which is a stable object to the right of the vertical wall $s = \mu(G^\bullet)$.  The correct nesting amounts to showing the center of $W(F,G^\bullet)$ lies to the right of the center of $W(G^\bullet,A)$, and this is equivalent to the displayed inequality.

As usual, there are two cases to consider based on whether $G^\bullet$ is destabilized horizontally or vertically.  First, suppose $G^\bullet$ is destabilized horizontally, with $\mu_j(G^\bullet)=\mu_\opt(G^\bullet)$, where $\ell\leq j <k$.  We know $\mu_\opt(F) = \mu_k(\cI_Z)$, and by purity of $Z$, $\mu_j(\cI_Z)\leq \mu_k(\cI_Z)$, so $$\frac{n-w_j}{j}+\frac{j-3}{2} \leq \frac{n}{k}+\frac{k-3}{2}.$$ Viewing this as a lower bound on $w_j$ then shows that $$\mu_j(G^\bullet)= \frac{1}{2}(j+k)+\frac{w_j}{k-j}-\frac{3}{2}\geq \frac{n}{k}+\frac{k-3}{2}= \mu_\opt(F).$$

If instead $G^\bullet$ is destabilized vertically with $\mu'_j(G^\bullet)=\mu_\opt(G^\bullet)$, we have an inequality $\Delta_j(F)\leq \Delta_i(F)$, where $\ell' \leq j  < i$.  As in Lemma \ref{rankZeroNest1}, we get the inequality $$w_j' \geq \frac{1}{2}(j-i)(i+j-k)+\frac{n}{k}(i-j)+w_i',$$ which implies $$\mu'_j(G^\bullet) = \frac{1}{2}(j+i)+\frac{w_j'-w_i'}{i-j}-\frac{3}{2} \geq \frac{n}{k}+\frac{k-3}{2} = \mu_\opt(F)$$ as required.
\end{proof}

\subsection{Nesting for rank $-1$ objects} Finally, we consider a nontrivial rank $-1$ object $$F^\bullet := \OO_{\P^2}(-k)\oplus \OO_{\P^2}(-i) \to \cI_Z$$ such that $r(D(Z)) = k$ and $c(D(Z))=i$.  Without loss of generality, the conjectured destabilizing sequence is $$ 0 \to E(-j) \to F^\bullet \to G^\bullet \to 0,$$ where $E$ is the rank $0$ object $\cI_{W_j\subset (k-j)L}$ and $G^\bullet$ is the rank $-1$ object $$G^\bullet := \OO_{\P^2}(-j)\oplus \OO_{\P^2}(-i) \to \cI_{Z_j}.$$ Here $\ell \leq j < r(D(Z))$, where $\ell$ is the number of full rows in $D(Z)$.

\begin{lemma}\label{rankMinusOneNest1}
The twisted rank $0$ object $E(-j)$ is destabilized along a wall nested in $W(E(-j),F^\bullet)$, and $$\Delta_\opt(E)\leq \Delta_\opt(F^\bullet).$$ (We have $\mu_\opt(E)+j=\mu_\opt(F^\bullet)$.)
\end{lemma}
\begin{proof}
If $E$ is destabilized by the rank one object $A$, then $E(-j)$ is destabilized by $A(-j)$.  The correct nesting of the walls $W(E(-j),A(-j))$ and $W(E(-j),F^\bullet)$ is equivalent to the inequality.  If $A$ corresponds to the vertical slope $\mu'_m(E)$, then one uses the inequality $\mu_j(F^\bullet) \leq \mu'_m(F^\bullet)$ to conclude the required inequality on discriminants.
\end{proof}

\begin{lemma}\label{rankMinusOneNest2}
If the rank $-1$ object $G^\bullet$ is not trivial, then it is destabilized along a wall nested in $W(G^{\bullet},F^\bullet)$, and $$\mu_\opt(G^\bullet)\geq \mu_\opt(F^\bullet).$$
\end{lemma}
\begin{proof}
Equivalence between nesting and the inequality is easy.  For the inequality, either $G^\bullet$ is destabilized horizontally or vertically by a rank $0$ object.  In case it is destabilized horizontally, say with $\mu_\opt(G^\bullet) = \mu_m(G^\bullet)$ (where $m\geq \ell(Z_j)=\ell(Z)$), the inequality follows from the inequality $\mu_m(F^\bullet)\geq \mu_j(F^\bullet)$.  On the other hand, if $G^\bullet$ is destabilized vertically with $\mu_\opt(G^\bullet) = \mu'_m(G^\bullet),$ we have $m\geq \ell'(Z_j)\geq \ell'(Z)$ and therefore an inequality $\mu'_m(F^\bullet)\geq \mu_j(F^\bullet)$, allowing us to conclude $\mu_\opt(G^\bullet)\geq \mu_\opt(F^\bullet)$.
\end{proof}

\section{Interpolation for monomial objects}\label{sec-interpolate}

We conclude the paper by solving the interpolation problem for a monomial object.  

\begin{theorem}\label{mainThm}
Let $F$ be a monomial object, and let $\zeta = (r,\mu,\Delta)$ be a Chern character such that $\chi(\zeta^*,F)=0$.  Suppose $r$ is sufficiently large and divisible, and let $E\in \cP(\zeta)$ be a general prioritary bundle.  Let $\mu_\opt(F),\Delta_\opt(F)$ be the combinatorial invariants defined in Section \ref{sec-invariants}.
\begin{enumerate}
\item If $F$ has rank $1$, then $E\te F$ is acyclic if and only if $\mu \geq \mu_\opt(F)$.  That is, $\mu_{\min}^\perp(F) = \mu_\opt(F)$.

\item If $F$ has rank $0$, then $E\te F$ is acyclic if and only if $\Delta \geq \Delta_\opt(F), $ i.e. $\Delta_{\min}^\perp(F) = \Delta_\opt(F)$.

\item If $F$ has rank $-1$, then $E\te F$ is acyclic if and only if $\mu \leq \mu_\opt(F).$
\end{enumerate}

\end{theorem}

The hard part of the theorem is to show that acyclicity holds.  To establish acyclicity, let $\zeta_\opt(F) = (r, \mu_\opt(F),\Delta_\opt(F))$ where $r$ is sufficiently large and divisible, and suppose $F$ is destabilized as $$0\to A(-j) \to F \to B \to 0,$$  where $j$ is chosen so that $A$ is a monomial object.  By induction on the complexity of a monomial object, we assume the theorem is known for $A$ and $B$ if they are nontrivial.  

For each type of monomial object, we must first show that $E\te F$ is acyclic for a general $E\in \cP(\zeta_\opt(F))$.  

\begin{lemma}
If $F$ is a rank $1$ monomial object, then $E \te F$ is acyclic for a general $E\in \cP(\zeta_\opt(F))$.
\end{lemma}
\begin{proof}
If $A = \OO_{\P^2}$ is trivial, then $E(-j)$ is acyclic by Theorem \ref{prioritaryThm}.  
 
Suppose $A$ is nontrivial.  By Lemma \ref{rankOneNest1}, we have $$\mu_\opt(A)+j \leq \mu_\opt(F),$$ so $\mu(E(-j)) \geq \mu_\opt(A)$.  As $\chi(E(-j)\te A)=0$, we find by induction that $E\te A(-j)$ is acyclic if $E$ is general and its rank is sufficiently large and divisible.

Similarly, by Lemma \ref{rankOneNest2} we have $$\Delta_\opt(B) \leq \Delta_\opt(F).$$ We have $\chi(E\te B) = 0$, so again by induction $E\te B$ is acyclic for general $E$ with sufficiently large and divisible rank.

Increasing the rank to a common multiple if necessary, we conclude $E \te F$ is acyclic for a general $E \in \cP(\zeta_\opt(F))$.  
\end{proof}

The same statement is easily proved in the rank $0$ case using the results of \S \ref{ss-rank0nesting}.  We omit the proof, which follows the same formal framework.

In the rank $1$ case, it follows that $\mu_{\min}^\perp(F) \leq \mu_\opt(F)$. Likewise, in the rank $0$ case, we have $\Delta_{\min}^\perp(F) \leq \Delta_\opt(F)$.  However, in the rank $-1$ case we do not have an analogue of Theorems \ref{bumpingUpThm} or \ref{rank0bumpUp}, so a separate argument must be given to prove acyclicity.  Due to the next result and Serre duality, the rank $1$ and rank $-1$ interpolation problems are in fact equivalent.

\begin{proposition}
The derived dual of a rank $-1$ monomial object is, up to twists and shifts, a rank $1$ monomial object.  
\end{proposition}
\begin{proof}
Consider the rank $-1$ monomial object $$F^\bullet := (\OO_{\P^2}(-k) \oplus \OO_{\P^2}(-i)\to \cI_Z),$$ and look at the generators $x^{a_1},x^{a_2}y^{b_2},\ldots,y^{b_r}$ of the ideal $\cI_Z$.  By definition, we have $a_1 = i$ and $b_r=k$.  Then the minimal free resolution of $\cI_Z$ looks like $$0\to \bigoplus_{j=1}^{r-1}\OO(-a_j-b_{j+1}) \to \OO(-i)\oplus \bigoplus_{j=2}^{r-1} \OO(-a_j-b_j) \oplus \OO(-k)\to \cI_Z\to 0.$$ Then $F^\bullet$ is quasi-isomorphic to the complex   $$\bigoplus_{j=1}^{r-1}\OO(-a_j-b_{j+1}) \to \bigoplus_{j=2}^{r-1} \OO(-a_j-b_j)$$ of locally free sheaves, sitting in degrees $-1$ and $0$.  The derived dual is the complex $$G^\bullet: = \left(\bigoplus_{j=2}^{r-1} \OO(a_j+b_j)\to \bigoplus_{j=1}^{r-1}\OO(a_j+b_{j+1})\right)$$ sitting in degrees $0$ and $1$.  The matrix in the minimal free resolution of $\cI_Z$ is $$\left( \begin{array}{cccc} y^{b_2-b_1} \\ -x^{a_1-a_2} & y^{b_3-b_2}\\ & -x^{a_2-a_3} & \ddots \\ &&& y^{b_r-b_{r-1}}\\ &&& -x^{a_{r-1}-a_r} \end{array}\right),$$
while the matrix in $G^\bullet$ is obtained by deleting the first and last rows of this matrix and taking the transpose.  If $Z'$ is the monomial scheme cut by the $(r-2)\times (r-2)$ minors of the matrix in $G^\bullet$, then $G^\bullet$ is quasi-isomorphic to $\cI_{Z'}(i+k)[-1]$.
\end{proof}

\begin{remark}
Preserve the notation from the proof.  There is a simple description of the scheme $Z'$ in terms of block diagrams.  Draw the block diagram of $Z$ as a subset of the rectangle with $k$ rows and $i$ columns.  Take the complement of the block diagram of $Z$ in this rectangle, and rotate the resulting configuration of blocks by $180^\circ$.  The corresponding block diagram is the block diagram of $Z'$ (see Figure \ref{fig-rotate}).  Using this interpretation of the dual, it is easy to check $$\mu_\opt(F^\bullet)= -\mu_\opt(\cI_{Z'})+i+k-3,$$ so the invariant $\mu_\opt$ transforms appropriately under Serre duality.
\end{remark}

 \begin{figure}[htbp]
\begin{center}
\input{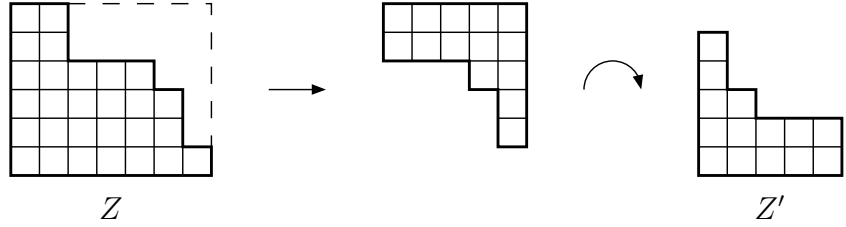}
\end{center}
\caption{The block diagram interpretation of the dual of a rank $-1$ monomial object.}
\label{fig-rotate}
\end{figure}

We have now established acyclicity in every case where it is supposed to hold.  We must still show acyclicity cannot hold in any other case.  By duality, it suffices to consider the rank $1$ and rank $0$ cases. 
       
\subsection{Dual curves for rank $1$ monomial objects}  Let $\cI_Z$ be a rank $1$ monomial object.  We have seen that a general $E\in \cP(\zeta_\opt(\cI_Z))$ has $E\te \cI_Z$ acyclic.  Correspondingly, there is a divisor $$D_E = c_1(E) H - \frac{r(E)}{2} B\sim \mu(E)H-\frac{1}{2}B$$ on $\P^{2[n]}$ such that $Z$ is not in the stable base locus of $D_E$.  Denote by $S^- \subset \P^{2[n]}$ the stable base locus of a divisor class $(\mu(E)-\epsilon)H-\frac{1}{2}B$ with $\epsilon>0$ sufficiently small.  If $\alpha$ is a curve on $\P^{2[n]}$ with $\alpha \cdot D_E = 0$, then $\alpha \subset S^{-}$.  To show that $\mu_{\min}^\perp(\cI_Z) = \mu_{\opt}(\cI_Z)$, it suffices to show that $Z\in S^-$.  

Suppose $\cI_Z$ is destabilized horizontally by a sequence $$0\to A(-j)\to \cI_Z \to B\to 0,$$ where $j$ is chosen so that $A$ is a rank $0$ monomial object.  By looking at other extensions $$0\to A(-j)\to \cI_{Z'} \to B\to 0,$$ (where $Z' \in \P^{2[n]}$ is not necessarily monomial) we can obtain a curve $\alpha$ in $\P^{2[n]}$ which passes through $Z$ and has $\alpha \cdot D_E=0$.  Note that since $E\te A(-j)$ and $E\te B$ are acyclic, any such ideal sheaf automatically has $E\te \cI_{Z'}$ acyclic, and thus $Z'$ is not in the stable base locus of $D_E$.  We must check that there is a complete curve of such extensions.

To do this, consider the block diagram $D_Z$ of $Z$.  Since $\mu_j(\cI_Z)\geq \mu_{j+1}(\cI_Z)$, the $(j+1)$st row of $D_Z$ must be shorter than the $j$th row of $D_Z$.  Then the block diagram of $Z_j$ has more ``full'' columns of length $j$ than the block diagram of $W_j$ has columns.  It follows that there is a monomial scheme $Z'$ whose block diagram is obtained from the block diagram of $Z$ by removing one box from each of the first $j$ rows.  For any $p = [u:v]\in L = \P^1$, let $Y_p \subset \P^2$ be the scheme given by the complete intersection $$Y_p = jL \cap \{v x - u z = 0\}.$$ For all $p\neq [0:1]$, the scheme $Z'\cup Y_p$ has degree $n$, and these schemes form a flat family over $L$ whose limit over $p = [0:1]$ is $Z$.  Note that, as $p$ varies, the ideal sheaf $\cI_{(Z'\cup Y_p) \cap jL \subset jL}$ has constant isomorphism type since $\cI_{Y_p\subset jL} = \OO_{jL}(-1)$.  Thus the ideal sheaf of each scheme $Z'\cup Y_p$ can be realized as an extension of $B$ by $A(-j)$, and we conclude that these schemes all lie in $S^{-}$.

\subsection{Optimality of interpolation for rank $0$ monomial objects}  Let $F = \cI_{Z\subset kL}$ be a rank $0$ monomial object.  Rather than produce a dual curve we will content ourselves with showing no bundle $E$ with $\Delta(E) < \Delta_\opt(F)$ satisfies interpolation.  Note that a similar argument can be given in the rank $1$ case as well.
 
Consider the destabilizing sequence $$0\to A(-i) \to F \to B^\bullet \to 0,$$ where $A$ is a rank $1$ monomial object and $B^\bullet$ is a rank $-1$ monomial object.  Suppose $E$ has $E\te F$ acyclic and $\Delta(E) < \Delta_\opt(F)$.  By Proposition \ref{rank0stable}, $E$ is semistable.  We have $\chi(F\te E) = 0$ and $\chi(A(-i) \te E)  >0$ since the point $(\mu(E),\Delta(E))$ lies below the parabola $\chi(A(-i)\te E) = 0$ in the $(\mu,\Delta)$-plane.  If the derived dual of $B^\bullet$ is written as $\cI_{Z'}(i+k)[-1]$, then by Serre duality $$R^{-1}\Gamma(B^\bullet \te^L E) = H^2(E^* \te \cI_{Z'}(i+k-3))^*.$$ This space is easily seen to vanish by the semistability of $E$.  Likewise, $H^2(A(-i)\te E)=0$, and from $\chi(A(-i)\te E)>0$ we deduce $h^0(A(-i)\te E)>0$.  But then $h^0(F\te E)\neq 0$, contradicting the acyclicity of $F\te E$.

\bibliographystyle{plain}

\end{document}